\documentclass[11pt,leqno]{article}
\usepackage[margin=1in]{geometry} 
\usepackage{amssymb,amsfonts,amsmath,bbm,mathrsfs,stmaryrd,mathtools,extarrows}
\usepackage{xcolor}
\usepackage{url}

\usepackage[shortlabels]{enumitem}
\usepackage{tensor}

\usepackage{xr}
\usepackage[colorlinks,
linkcolor=black!75!red,
citecolor=blue,
pdftitle={},
pdfproducer={pdfLaTeX},
pdfpagemode=None,
bookmarksopen=true,
bookmarksnumbered=true,
backref=page]{hyperref}

\usepackage{tikz}
\usetikzlibrary{arrows,calc,decorations.pathreplacing,decorations.markings,intersections,shapes.geometric,through,fit,shapes.symbols,positioning,decorations.pathmorphing}

\usepackage{braket}

\usepackage[a-1b]{pdfx}

\usepackage[amsmath,thmmarks,hyperref]{ntheorem}
\usepackage{cleveref}

\creflabelformat{enumi}{#2#1#3}

\crefname{section}{Section}{Sections}
\crefformat{section}{#2Section~#1#3} 
\Crefformat{section}{#2Section~#1#3} 

\crefname{subsection}{\S}{\S\S}
\AtBeginDocument{%
  \crefformat{subsection}{#2\S#1#3}%
  \Crefformat{subsection}{#2\S#1#3}%
}

\crefname{subsubsection}{\S}{\S\S}
\AtBeginDocument{%
  \crefformat{subsubsection}{#2\S#1#3}%
  \Crefformat{subsubsection}{#2\S#1#3}%
}

\theoremstyle{plain}

\newtheorem{lemma}{Lemma}[section]
\newtheorem{proposition}[lemma]{Proposition}
\newtheorem{corollary}[lemma]{Corollary}
\newtheorem{theorem}[lemma]{Theorem}

\theoremstyle{plain}
\theoremnumbering{Alph}
\newtheorem{theoremN}{Theorem}

\theoremstyle{plain}
\theorembodyfont{\upshape}
\theoremsymbol{\ensuremath{\blacklozenge}}

\newtheorem{definition}[lemma]{Definition}
\newtheorem{example}[lemma]{Example}

\newtheorem{remark}[lemma]{Remark}
\newtheorem{remarks}[lemma]{Remarks}

\newtheorem{notation}[lemma]{Notation}

\newtheorem{recollection}[lemma]{Recollection}

\crefname{definition}{definition}{definitions}
\crefformat{definition}{#2definition~#1#3} 
\Crefformat{definition}{#2Definition~#1#3} 

\crefname{ex}{example}{examples}
\crefformat{example}{#2example~#1#3} 
\Crefformat{example}{#2Example~#1#3} 

\crefname{exs}{example}{examples}
\crefformat{examples}{#2example~#1#3} 
\Crefformat{examples}{#2Example~#1#3} 

\crefname{remark}{remark}{remarks}
\crefformat{remark}{#2remark~#1#3} 
\Crefformat{remark}{#2Remark~#1#3} 

\crefname{remarks}{remark}{remarks}
\crefformat{remarks}{#2remark~#1#3} 
\Crefformat{remarks}{#2Remark~#1#3} 

\crefname{convention}{convention}{conventions}
\crefformat{convention}{#2convention~#1#3} 
\Crefformat{convention}{#2Convention~#1#3} 

\crefname{notation}{notation}{notations}
\crefformat{notation}{#2notation~#1#3} 
\Crefformat{notation}{#2Notation~#1#3} 

\crefname{table}{table}{tables}
\crefformat{table}{#2table~#1#3} 
\Crefformat{table}{#2Table~#1#3}

\crefname{lemma}{lemma}{lemmas}
\crefformat{lemma}{#2lemma~#1#3} 
\Crefformat{lemma}{#2Lemma~#1#3} 

\crefname{proposition}{proposition}{propositions}
\crefformat{proposition}{#2proposition~#1#3} 
\Crefformat{proposition}{#2Proposition~#1#3} 

\crefname{corollary}{corollary}{corollaries}
\crefformat{corollary}{#2corollary~#1#3} 
\Crefformat{corollary}{#2Corollary~#1#3} 

\crefname{theorem}{theorem}{theorems}
\crefformat{theorem}{#2theorem~#1#3} 
\Crefformat{theorem}{#2Theorem~#1#3} 

\crefname{theoremN}{theorem}{theorems}
\crefformat{theoremN}{#2theorem~#1#3} 
\Crefformat{theoremN}{#2Theorem~#1#3} 

\crefname{enumi}{}{}
\crefformat{enumi}{#2#1#3}
\Crefformat{enumi}{#2#1#3}

\crefname{assumption}{assumption}{Assumptions}
\crefformat{assumption}{#2assumption~#1#3} 
\Crefformat{assumption}{#2Assumption~#1#3} 

\crefname{construction}{construction}{Constructions}
\crefformat{construction}{#2construction~#1#3} 
\Crefformat{construction}{#2Construction~#1#3} 

\crefname{recollection}{recollection}{Recollections}
\crefformat{recollection}{#2recollection~#1#3} 
\Crefformat{recollection}{#2Recollection~#1#3} 

\crefname{equation}{}{}
\crefformat{equation}{(#2#1#3)} 
\Crefformat{equation}{(#2#1#3)}

\numberwithin{equation}{section}

\theoremstyle{nonumberplain}
\theoremsymbol{\ensuremath{\blacksquare}}

\newtheorem{proof}{Proof}

\newcommand\bC{{\mathbb C}}

\newcommand\bG{{\mathbb G}}
\newcommand\bH{{\mathbb H}}

\newcommand\bP{{\mathbb P}}

\newcommand\bZ{{\mathbb Z}}

\newcommand\cE{{\mathcal E}}
\newcommand\cF{{\mathcal F}}
\newcommand\cG{{\mathcal G}}
\newcommand\cH{{\mathcal H}}

\newcommand\cK{{\mathcal K}}
\newcommand\cL{{\mathcal L}}
\newcommand\cM{{\mathcal M}}
\newcommand\cN{{\mathcal N}}
\newcommand\cO{{\mathcal O}}

\newcommand\cQ{{\mathcal Q}}

\newcommand\cS{{\mathcal S}}
\newcommand\cT{{\mathcal T}}
\newcommand\cU{{\mathcal U}}
\newcommand\cV{{\mathcal V}}

\newcommand\fF{{\mathfrak F}}

\newcommand\ol{\overline}

\DeclareMathOperator{\id}{id}
\DeclareMathOperator{\im}{im}
\DeclareMathOperator{\spn}{span}

\DeclareMathOperator{\Sec}{\mathrm{Sec}}
\DeclareMathOperator{\End}{\mathrm{End}}
\DeclareMathOperator{\Hom}{\mathrm{Hom}}

\DeclareMathOperator{\Aut}{\mathrm{Aut}}
\DeclareMathOperator{\Pic}{\mathrm{Pic}}

\DeclareMathOperator{\Quot}{\mathrm{Quot}}
\DeclareMathOperator{\Coh}{\mathrm{Coh}}

\def\End{\operatorname {End}}
\def\Ext{\operatorname {Ext}}

\def\rk{\operatorname {rk}}
\def\gr{\operatorname {gr}}
\def\md{\operatorname {mid}}

\newcommand{\cat}[1]{\textsc{#1}}

\newcommand{\qedhere}{\mbox{}\hfill\ensuremath{\blacksquare}}

\renewcommand{\square}{\mathrel{\Box}}

\newcommand{\xrightarrowdbl}[2][]{%
  \xrightarrow[#1]{#2}\mathrel{\mkern-14mu}\rightarrow
}

\title{Bundle-extension inverse problems over elliptic curves}
\author{Alexandru Chirvasitu}

\begin{document}

\date{}

\newcommand{\Addresses}{{
  \bigskip
  \footnotesize

  \textsc{Department of Mathematics, University at Buffalo}
  \par\nopagebreak
  \textsc{Buffalo, NY 14260-2900, USA}  
  \par\nopagebreak
  \textit{E-mail address}: \texttt{achirvas@buffalo.edu}


}}

\maketitle

\begin{abstract}
  We prove a number of results to the general effect that, under obviously necessary numerical and determinant constraints, ``most'' morphisms between fixed bundles on a complex elliptic curve produce (co)kernels which can either be specified beforehand or else meet various rigidity constraints. Examples include: (a) for indecomposable $\mathcal{E}$ and $\mathcal{E'}$ with slopes and ranks increasing strictly in that order the space of monomorphisms whose cokernel is semistable and maximally rigid (i.e. has minimal-dimensional automorphism group) is open dense; (b) for indecomposable $\mathcal{K}$, $\mathcal{E}$ and stable $\mathcal{F}$ with slopes increasing strictly in that order and ranks and determinants satisfying the obvious additivity constraints the space of embeddings $\mathcal{K}\to \mathcal{E}$ whose cokernel is isomorphic to $\mathcal{F}$ is open dense; (c) the obvious mirror images of these results; (d) generalizations weakening indecomposability to semistability + maximal rigidity; (e) various examples illustrating the necessity of the assorted assumptions. 
\end{abstract}

\noindent {\em Key words:
  Chern class;
  Harder-Narasimhan filtration;
  Harder-Narasimhan polygon;
  Jordan-H\"older filtration;
  Quot scheme;
  associated grading;
  bounded derived category;
  degree;
  determinant;
  divisor;
  elliptic curve;
  moduli space;
  pencil;  
  quadric;
  rank;
  secant variety;
  slope;
  spherical object;
  twist functor;
  vector bundle

}

\vspace{.5cm}



\section*{Introduction}

We work over an algebraically closed field $\Bbbk$ of characteristic 0 throughout (the reader may as well assume $\Bbbk=\bC$), with {\it vector bundles} \cite[\S B.3]{Fulton-2nd-ed-98} over elliptic curves, often conflating them with their corresponding locally free sheaves of sections \cite[\S B.3.2]{Fulton-2nd-ed-98}. 

The problems studied below arise naturally in the course of examining the Poisson geometry of bundle-extension spaces
\begin{equation*}
  \bP\Ext^1(\cF,\cO)
  :=
  Ext^1(\cF,\cO)^{\times}/\Bbbk^{\times}
  ,\quad
  \cF\text{ stable}
\end{equation*}
in \cite{FO98} and numerous follow-ups \cite{pl98,HP1,HP2,hp_23,2306.14719v1,2310.05284v1,00-leaves_xv3,ch_sympl-qnk_xv1}, {\it (semi)stability} being understood as usual (e.g. \cite[Definition 10.20]{muk-invmod}), as having no proper non-zero subbundles of $\ge$ (respectively strictly larger) {\it slope} (= degree/rank; cf. \Cref{se:prel} for brief background recollections). 

One comparatively elementary issue, in that context, addressed in \cite[]{00-leaves_xv3} for line bundles and \cite[Theorem 2.4]{ch_sympl-qnk_xv1} for more general stable $\cF$, is that of determining which middle terms occur as middle terms of extensions
\begin{equation*}
  0\to
  \cO
  \lhook\joinrel\xrightarrow{\quad}
  \bullet
  \xrightarrowdbl{\quad}
  \cF
  \to 0
\end{equation*}
classified by appropriately non-degenerate elements of $\Ext^1(\cF,\cO)$. The inverse problems of the present paper's title are of that nature: fitting preselected vector bundles $\cK$, $\cE$ and $\cF$ into exact sequences
\begin{equation}\label{eq:kef}
  0\to
  \cK
  \lhook\joinrel\xrightarrow{\quad}
  \cE
  \xrightarrowdbl{\quad}
  \cF
  \to 0,
\end{equation}
under various sets of sufficient (and, one hopes, not onerous) conditions. In that direction, a slightly simplified \Cref{th:prescribedseq} reads:

\begin{theoremN}\label{thn:3bdls}
  Let $\cK$, $\cE$ and $\cF$ be indecomposable bundles, with ranks and determinants satisfying
  \begin{equation}\label{eq:slopesrksdets}
    \begin{cases}
      \mu(\cK) < \mu(\cF)\\
      \rk \cE = \rk \cK+\rk \cF\\
      \det\cE \cong \det\cF\otimes \det\cK      
    \end{cases}
  \end{equation}  
  \begin{enumerate}[(1),wide]
  \item\label{item:th:prescribedseq:fstab} If $\cF$ is stable then the space of embeddings $\cK\lhook\joinrel\to \cE$ fitting into an exact sequence \Cref{eq:kef} is open dense in $\Hom(\cK,\cE)$.

  \item\label{item:th:prescribedseq:kstab} Dually, if $\cK$ is stable then the space of epimorphisms $\cE\xrightarrowdbl{}\cF$ that can be completed to extensions \Cref{eq:kef} is open dense in $\Hom(\cE,\cF)$. 
  \end{enumerate}
  In particular, if either $\cF$ or $\cK$ is stable then extensions \Cref{eq:kef} exist.  \qedhere
\end{theoremN}

\Cref{th:prescribedseq} differs from the simplified statement just given in relaxing indecomposability: while the latter does matter (\hyperref[res:notdense]{Remarks~\ref*{res:notdense}}), the more relevant constraint is symmetry paucity (a theme that will recur): \Cref{thn:3bdls} goes through for what \Cref{res:pr:maxnondeg:post}\Cref{item:res:pr:maxnondeg:post:fixindec} refers to as {\it maximally asymmetric} semistable bundles (or {\it basic}: \Cref{res:pr:maxnondeg:post}\Cref{item:res:pr:maxnondeg:post:basic}): those which
\begin{itemize}[wide]
\item given their ranks and degrees, have minimal-dimensional automorphism group;

\item or equivalently, break up as direct sums of indecomposable bundles with no common stable subquotients. 
\end{itemize}
The equivalence follows from \Cref{eq:selfexts} below (and is also proven in \cite[Lemma 5.6(i)]{2306.14719v1} (where said bundles are termed {\it quasi indecomposable}).

\Cref{thn:3bdls} is suggestive of another, related track: results to the effect that ``most'' morphisms between two given bundles exhibit as little degeneracy as one could possibly expect. Thus (\Cref{th:maxnondeg}):

\begin{theoremN}\label{thn:2bdls}
  For indecomposable bundles $\cE$ and $\cE'$ with slopes $\mu(\cE)<\mu(\cE')$ the following subspace is open dense in $\Hom(\cE,\cE')$ depending on the relative ordering of the ranks:
  \begin{itemize}[wide]
  \item that of epimorphisms $\cE\xrightarrowdbl{}\cE'$ if $\rk\cE>\rk\cE'$;

  \item that of torsion-free-cokernel monomorphisms $\cE\lhook\joinrel\xrightarrow{}\cE'$ if $\rk\cE<\rk\cE'$;

  \item that of monomorphisms $\cE\lhook\joinrel\xrightarrow{}\cE'$ (perhaps with non-vanishing, but automatically torsion cokernel) if $\rk\cE=\rk\cE'$.  \qedhere
  \end{itemize}
\end{theoremN}

Here too indecomposability does play a role (\hyperref[res:pr:maxnondeg:post]{Remarks~\ref*{res:pr:maxnondeg:post}}), but can be loosened to maximal asymmetry (\Cref{cor:th:maxnondeg:dec}). 

Or, in the same spirit, and coming full circle (for \Cref{th:asym} in fact implies the stronger counterpart \Cref{th:prescribedseq} to \Cref{thn:3bdls} above):

\begin{theoremN}\label{thn:maxss}
  Let $\cE$ and $\cE'$ be unequal-rank maximally asymmetric semistable bundles with slopes $\mu(\cE)<\mu(\cE')$. The space of morphisms $\cE\to \cE'$ that are
  \begin{itemize}[wide]
  \item surjective with semistable maximally asymmetric kernel if $\rk\cE>\rk\cE'$;

  \item or injective with semistable maximally asymmetric cokernel if $\rk\cE<\rk\cE'$
  \end{itemize}
  is open dense in $\Hom(\cE,\cE')$.  \qedhere
\end{theoremN}

\subsection*{Acknowledgements}

The material has benefited from valuable input from M. Aprodu,  R. Kanda, M. Fulger, S. P. Smith and V. Srinivas. This work is partially supported by NSF grant DMS-2001128. 

\section{Preliminaries}\label{se:prel}

The {\it projectivization} operator $\bP$ applies to vector spaces $V$ (producing the space $\bP V$ of lines therein), to bundles $\cE$ as in \cite[\S B.5.5]{Fulton-2nd-ed-98}, producing the bundle $\bP \cE$ whose fiber at the point $x$ of the base is the projectivized space $\bP \cE_x$, and so on. 


\begin{notation}\label{not:miscinit}
  Frequently-used symbols and notions include ($\cE$ being a coherent sheaf on $E$):
  \begin{itemize}[wide]
  \item the {\it rank} $\rk \cE\in \bZ_{\ge 0}$ (\cite[Exercise II.6.10]{hrt}, \cite[Definition 4.4]{BB-vb});

  \item the {\it degree} $\deg\cE$ \cite[Exercise II.6.12]{hrt};

  \item the {\it Euler characteristic} \cite[Exercise III.5.1]{hrt}
    \begin{equation*}
      \chi(\cE):=\dim H^0(\cE) - \dim H^1(\cE);
    \end{equation*}
    
  \item the {\it slope}
    \begin{equation*}
      \begin{aligned}
        \mu(\cE)
        &:=\frac{\chi(\cE)}{\rk \cE}
          \quad\text{following \cite[Definition 4.6]{BB-vb}}\\
        &=\frac{\deg\cE}{\rk\cE}
          \quad\text{on elliptic curves},
      \end{aligned}
    \end{equation*}
    the last equality following from {\it Riemann-Roch} (\cite[Theorem 14.1]{3264}, \cite[Theorem 4.13]{BB-vb}, etc.);

  \item the {\it charge} \cite[post Example 4.16]{BB-vb}
    \begin{equation*}
      \zeta(\cE) = (\rk \cE,\ \deg\cE)\in \bZ_{\ge 0}\times \bZ;
    \end{equation*}

  \item the {\it determinant} (\cite[Exercise II.6.11]{hrt}, generalized)
    \begin{equation*}
      \det\cE:=
      \bigotimes_{i=0}^s
      \left(\bigwedge^{r_i}\cE_i\right)^{(-1)^{i}}
    \end{equation*}
    for a resolution
    \begin{equation}\label{eq:resole}
      0\to
      \cE_s
      \xrightarrow{\quad}
      \cE_{s-1}
      \xrightarrow{\quad}
      \cdots
      \xrightarrow{\quad}
      \cE_1
      \xrightarrow{\quad}
      \cE_0
      \xrightarrow{\quad}
      \cE
      \to 0
    \end{equation}
    of the coherent sheaf $\cE$ on a smooth projective variety by bundles $\cE_i$ of respective ranks $r_i$. This is also the {\it first Chern class} \cite[\S\S 5.2, 5.3]{3264} $c_1(\cE)$ of said coherent sheaf, regarding $c_1$ as a line bundle rather than a divisor \cite[Proposition 1.30]{3264}. 
    
    More generally, \cite[\S 14.2.1]{3264} defines (resolution-independent) higher Chern classes $c_i(\cE)$ in similar fashion in the same context (coherent sheaves on smooth projective varieties) via a resolution \Cref{eq:resole}.

  \item the {\it Abel-Jacobi map} (e.g. \cite[Chapter I, equation (3.2)]{acgh1}, in one version)
    \begin{equation}\label{eq:aj}
      \left(\text{divisor }\sum_i n_i(z_i)_i\right)
      \xmapsto{\quad\sigma\quad}
      \sum_i n_iz_i
      \in E;
    \end{equation}
    we apply the notation to line bundles as well: $\sigma(\cL(D)):=\sigma(D)$. 
    
  \item For a degree-$d$ divisor $D$ set
    \begin{equation}\label{eq:ldd}
      \tensor[^{d;\sigma(D)}]{\cL}{}
      \quad
      \left(\text{or only $\tensor[^{\sigma(D)}]{\cL}{}$ if the degree is understood}\right)
      \quad
      :=
      \quad
      \cO(D).
    \end{equation}

  \item The {\it associated graded bundle} $\gr \cE$ is the direct sum of the subquotients of a {\it Jordan-H\"older filtration} (\cite[Proposition 5.3.7 and surrounding discussion]{lepot-vb}, \cite[\S 1.5]{HL10}): a flag
    \begin{equation*}
      0=\cE_0<\cE_1<\cdots<\cE_s = \cE,
    \end{equation*}
     maximally fine under the constraint that the subquotients $\cE_i/\cE_{i-1}$ be stable with slopes non-increasing in $i$ (in particular, the slopes are all equal to $\mu(\cE)$ when $\cE$ is semistable).

  \item We write
    \begin{equation*}
      \begin{aligned}
        E^{[r]}        
        &:=E^r/\text{(action of the symmetric group $S_r$)}\\
        &=\left\{\text{unordered $r$-tuples in }E\right\}
      \end{aligned}  
    \end{equation*}
    for the $r^{th}$ {\it symmetric power} of $E$ (\cite[p.18]{acgh1}, \cite[\S 1(1)]{CaCi93}, etc.).
  \end{itemize}
\end{notation}

One observation used extensively in the sequel, noted as \cite[Summary post Corollary 4.25, last bullet point]{BB-vb}, is that for indecomposable bundles $\cE$ and $\cF$ on an elliptic curve we have
\begin{equation}\label{eq:degrk}
  \begin{aligned}
    \mu(\cE)<\mu(\cF)
    &\xRightarrow{\quad}
      \begin{cases}
        \dim\Hom(\cE,\cF) &= \deg\cF\cdot \rk\cE-\deg\cE\cdot \rk \cF\\
        \Ext^1(\cE,\cF) &=\{0\}
      \end{cases}\\
    \mu(\cE)>\mu(\cF)
    &\xRightarrow{\quad}
      \begin{cases}
        \Hom(\cE,\cF) &= \{0\}\\
        \dim\Ext^1(\cE,\cF) &= \deg\cF\cdot \rk\cE-\deg\cE\cdot \rk \cF.
      \end{cases}
  \end{aligned}
\end{equation}
More generally, by the additivity of both degrees \cite[post Definition 10.7]{muk-invmod} and ranks under direct sums, \Cref{eq:degrk} holds provided the requisite slope inequalities hold between every summand of $\cE$ in the {\it Harder-Narasimhan (HN for short) decomposition} or {\it filtration} \cite[Proposition 5.4.2]{lepot-vb}: generally the unique flag
\begin{equation}\label{eq:hnfilt}
  \begin{aligned}
    &0<\cE_1<\cdots<\cE_s=\cE
      ,\quad
      \overline{\cE}_i:=\cE_{i}/\cE_{i-1}\text{ semistable}\\
    &\mu_{max}(\cE)
      :=
      \mu\left(\overline{\cE}_1\right)
      \ 
      >
      \ 
      \mu\left(\overline{\cE}_2\right)
      \ 
      >
      \ 
      \cdots
      \ 
      >
      \
      \mu\left(\overline{\cE}_s\right)
      =:
      \mu_{min}(\cE)
  \end{aligned}  
\end{equation}
a filtration generally, a direct sum over elliptic curves.

\begin{remark}\label{re:genassocgr}
  The associated graded object $\gr \cE$ is usually considered for semistable $\cE$ (e.g \cite[Proposition 1.5.2]{HL10}). Given the {\it canonical} nature of the HN filtration though, it makes sense to extend the concept to arbitrary $\cE$:
  \begin{equation*}
    \gr\cE
    :=
    \bigoplus_i \gr\overline{\cE}_i
  \end{equation*}
  for the subquotients $\overline{\cE}_i$ of \Cref{eq:hnfilt}.   
\end{remark}

As a particular case of \Cref{eq:degrk} (upon setting $\cF:=\cO$), we remind the reader \cite[Lemma 17]{tu} that for semistable $\cE$
\begin{equation}\label{eq:tu_lemma-17}
  \begin{aligned}
    \deg\cE<0
    &\xRightarrow{\quad}
      \Gamma(\cE)=\{0\}\text{ and }\dim H^1(\cE)=\deg\cE\\
    \deg\cE>0
    &\xRightarrow{\quad}
      \dim\Gamma(\cE)=\deg\cE\text{ and }H^1(\cE)=\{0\}.
  \end{aligned}  
\end{equation}

It will be profitable to have available some shorthand notation for the rightmost expressions in \Cref{eq:degrk}:
\begin{equation}\label{eq:brake2f}
  \braket{\cE\to \cF}
  :=
  \deg\cF\cdot \rk\cE-\deg\cE\cdot \rk \cF.
\end{equation}

We record also the following estimate on hom-space sizes; it follows easily from \Cref{eq:degrk} and the additivity of $\braket{\bullet\to \bullet}$ in each variable with respect to short exact sequences. In other words: $\braket{\bullet\to\bullet}$ descends to a bilinear map
\begin{equation*}
  K(E)^2
  \xrightarrow{\quad\braket{\bullet\to\bullet}\quad}
  \bZ
\end{equation*}
on the {\it Grothendieck group} $K(E)$ of $E$ (\cite[Definition 1.3]{man_k}, \cite[Exercises II.6.10 and III.6.9 and Appendix A]{hrt}, etc.).

\begin{lemma}\label{le:smallhom}
  For bundles $\cE$ and $\cF$ on $E$
  \begin{equation*}
    \gr(\cE)\perp \gr(\cF)
    \quad
    \xRightarrow{\quad}
    \quad
    \dim\Hom(\cE,\cF)
    =
    \dim\Ext^1(\cF,\cE)
    \le \braket{\cE\to \cF}
  \end{equation*}
  (`$\perp$' meaning no common summands).  \qedhere
\end{lemma}

\begin{remark}\label{re:selfext}
  Contrast \Cref{eq:degrk} with the case of {\it equal}-slope (semistable) bundles.
  
  It follows from \cite[Theorem 10 and Lemma 24]{Atiyah} (or \cite[Propositions 14 and 21]{tu}) that an indecomposable $\cE$ of charge $(r,d)$ is an $h$-fold iterated self-extension $\tensor*[_h]{\cE}{^'}$ of a unique (up to isomorphism) stable charge-$(r',d')$ bundle $\cE'$ with
  \begin{equation*}
    h:=\gcd(r,d)
    ,\quad
    (r,d)
    =
    (hr',hd')
    \quad\text{so that $\gcd(r',d')$=1}.
  \end{equation*}
  As the notation $\tensor*[_h]{\cE}{^'}$ suggests (involving $h$ and $\cE'$ only), $\cE'$ also determines $\cE$ uniquely: there is only one such indecomposable self-extension for every $h$ and stable $\cE'$. Returning to \Cref{eq:degrk}, its (or rather one) analogue for stable equal-slope $\cE$ and $\cF$ is
  \begin{equation}\label{eq:selfexts}
    \begin{aligned}
      \forall h,\ell\in \bZ_{>0}
      ,\quad
      \dim\Hom(\tensor*[_h]{\cE}{},\ \tensor*[_{\ell}]{\cF}{})
      &=
        \dim\Ext^1(\tensor*[_h]{\cE}{},\ \tensor*[_{\ell}]{\cF}{})\\
      &=
        \delta_{\cE,\cF}\cdot\min(h,\ell)
        :=
        \begin{cases}
          \min(h,\ell)&\text{if }\cE\cong \cF\\
          0&\text{ otherwise}.
        \end{cases}
    \end{aligned}    
  \end{equation}
  The $\Hom$ and $\Ext$ versions are mutually equivalent by {\it Serre duality} \cite[Theorem 3.12]{Huy-FM}:
  \begin{equation*}
    \Ext^1(\square,\bullet)    
    \cong
    \Hom(\bullet,\square)^*
  \end{equation*}  
  For $\Hom$ and $\cE=\cF=\cO$ this is recorded explicitly as \cite[Lemma 17(i)]{Atiyah}, hence the general version by
  \begin{equation*}
    \Hom(\tensor*[_h]{\cE}{},\ \tensor*[_{\ell}]{\cF}{})
    \cong
    \Hom(\tensor*[_h]{\cO}{}\otimes \cE,\ \tensor*[_{\ell}]{\cO}{}\otimes \cF)
    \cong
    \Hom(\tensor*[_h]{\cO}{},\ \tensor*[_{\ell}]{\cO}{}\otimes \cF\otimes \cE^*)
  \end{equation*}
  and the decomposition \cite[Lemma 22 and Theorem 10]{Atiyah} of $\cF\otimes \cE^*$ as a direct sum of $\rk\cE\cdot \rk\cF=(\rk\cE)^2$ mutually non-isomorphic degree-0 line bundles, including $\cO$ among them precisely when $\cE\cong \cF$. 
\end{remark}

The following invariant refines the charge (the terminology follows \cite[\S 3]{shatz_dec} and \cite[\S 7]{ab_ym}). 

\begin{definition}\label{def:type}
  The {\it (HN) type} $\tau(\cE)$ of a bundle $\cE$ is the tuple
  \begin{equation*}
    \tau(\cE):=((\rk\cE_i,\ \deg\cE_i))_{i=1}^s
    ,\quad
    0<\cE_1<\cdots<\cE_s=\cE
    \text{ the HN filtration of }\cE.
  \end{equation*}

  The {\it HN polygon} $HNP(\cE)$ is the polygonal line connecting the origin $(0,0)$ in the plane and the {\it vertices} (i.e. members) of $\tau(\cE)$ consecutively in increasing $x$-coordinate order.

  HNPs (or types) are ordered as in \cite[Theorem 3]{shatz_dec} (or more explicitly, \cite[(7.6)]{ab_ym}): $HNP(\cE)\le HNP(\cE')$ (and similarly for the respective types) when $HNP(\cE)$ is placed under $HNP(\cE')$ in the plane. 
\end{definition}


\section{Bundle extensions with prescribed terms}\label{se:prescr-term}

\subsection{Generic extensions of asymmetric bundles}\label{subse:prescr-term}

\cite[Lemma 2.5]{ch_sympl-qnk_xv1}, to the effect that there are always non-zero morphisms from lower to higher-slope bundles, can be improved upon when those bundles are indecomposable.

\begin{definition}\label{def:maxnondeg}
  A morphism $\cE\to \cE'$ of vector bundles is {\it maximally non-degenerate} if it is
  \begin{itemize}[wide]
  \item epic if $\rk\cE\ge \rk \cE'$;
  \item and monic with torsion-free quotient $\cE'/\cE$ (so that that quotient is itself a vector bundle \cite[Lemma 5.2.1]{lepot-vb} when working over a smooth curve, as we are here) if $\rk\cE\le \rk \cE'$. 
  \end{itemize}
  We write
  \begin{equation*}
    \Hom_{max}(\cE,\cE')
    :=
    \left\{\cE\xrightarrow{f}\cE'\ |\ f\text{ maximally non-degenerate}\right\}. 
  \end{equation*}
  and also $\Hom_{\lhook\joinrel\to}(\cE,\cE')$ and $\Hom_{\twoheadrightarrow}(\cE,\cE')$ for the spaces of monomorphisms (possibly with torsion cokernel) and epimorphisms respectively. 
\end{definition}

For vector bundles (even more generally, coherent sheaves) $\cE$ and $\cE'$ on $E$ the $\Bbbk$-linear space
\begin{equation}\label{eq:bh}
  \bH
  :=
  \tensor[_{\cE\to\cE'}]{\bH}{}
  :=
  \Hom(\cE,\cE'),      
\end{equation}
regarded as an affine $\Bbbk$-scheme, represents the functor
\begin{equation*}
  \left(\text{$\Bbbk$-schemes }T\right)
  \xmapsto[]{\quad\text{{\it fpqc} sheaf \cite[\S 2.3.2]{vist_grothtop_2008-10-02}}\quad}
  \left(\text{morphisms }\cE_T\to \cE'_T\right)
  \in \cat{Set},
\end{equation*}
where $\cE_T$ and $\cE'_T$ are the pullbacks of $\cE$ and $\cE'$ respectively to $T\times E$. This is easily seen from, say, from \cite[Theorem 5.8 and its proof via Theorems 5.6 and 5.7]{MR2223407} (also \cite[Th\'eor\`eme 7.7.6 and Corollaire 7.7.8]{ega32}). In particular, there is a universal morphism
\begin{equation}\label{eq:univmor}
  \cE_{\bH}
  \xrightarrow{\quad\cat{univ}=\tensor*[_{\cE\to \cE'}]{\cat{univ}}{}\quad}
  \cE'_{\bH}
\end{equation}
on $\bH\times E$. If $\rk\cE\ne \rk\cE'$ then
\begin{equation*}
  \bH_{max}
  :=
  \begin{cases}
    \{\text{locus where $\cat{univ}$ is onto}\}
    =
    \{\text{onto }\cE\xrightarrowdbl{\quad}\cE'\}
    &\text{ if $\rk\cE>\rk\cE'$}\\
    \{\text{locus where $\cat{univ}$ is monic with torsion-free cokernel}\}    
    &\text{ if $\rk\cE<\rk\cE'$}\\
  \end{cases}  
\end{equation*}
is precisely the open subscheme $\Hom_{max}(\cE,\cE')$ of \Cref{th:maxnondeg}. When $\rk\cE>\rk\cE'$ the kernel of $\cat{univ}$ is a rank-$(\rk\cE-\rk\cE')$ bundle over the projective $\bH_{max}$-scheme $\bH_{max}\times E$, and a dual remark applies when $\rk\cE<\rk\cE'$. The first part of (the doubtless well-known) \Cref{le:semistabopen} follows from the open nature of
\begin{itemize}[wide]
\item surjectivity and injectivity (standard: locally a morphism $\cE\to \cE'$ is a matrix, and ranks of algebraic families of matrices are lower semicontinuous, per \cite[proof of Lemma 2.6.1]{lepot-vb} or \cite[Proposition 4.7(1)]{CKS4});
  
\item (semi)stability (\cite[Theorem 2.8(B)]{maruy_open} or \cite[\S 1, Theorem 4.2]{zbMATH06600951}) ;

\item and torsion-freeness (\cite[Proposition 2.1]{maruy_open} or \cite[Th\'eor\`eme 12.2.1(i)]{ega43}).
\end{itemize}
\cite[Proposition 10 and Corollary]{shatz_dec} prove the second. 

\begin{lemma}\label{le:semistabopen}
  For morphisms $f\in \Hom_{max}(\cE,\cE')$ between two bundles, set $\cat{kc}(f)$ to be either the kernel or cokernel of $f$, depending on whether the rank of $\cE'$ is smaller (respectively larger).
  \begin{enumerate}[(1),wide]
  \item\label{item:le:semistabopen:ss} The subspaces $\Hom_{\lhook\joinrel\to}(\cE,\cE')$, $\Hom_{max}(\cE,\cE')$ and 
    \begin{equation}\label{eq:hommaxss}
      \begin{aligned}        
        \Hom_{max\mid ss}(\cE,\cE')
        &:=
          \left\{f\in \Hom_{max}(\cE,\cE')\ |\ \cat{kc}(f)\text{ semistable}\right\}\\
        &\subseteq \Hom_{max}(\cE,\cE')
          \subseteq \Hom(\cE,\cE')
      \end{aligned}      
    \end{equation}
    are open.
    
  \item\label{item:le:semistabopen:type} The map
    \begin{equation*}
      \Hom_{max}(\cE,\cE')
      \ni
      f
      \xmapsto{\quad}
      \tau(\cat{kc}(f))
    \end{equation*}
    is upper semicontinuous for the Zariski topology on the domain and the ordering of \Cref{def:type} on the codomain. In particular, the loci of maximally non-degenerate morphisms with $\cat{kc}$ of given type are locally closed.  \qedhere  
  \end{enumerate}
\end{lemma}

And very much in the same general vein:

\begin{proposition}\label{pr:taulcl}
  For bundles $\cE$ and $\cE'$ the following subspaces of $\Hom(\cE,\cE')$ are all locally closed.

  \begin{enumerate}[(1),wide]
  \item\label{item:le:taulcl:tau} For any type $\tau$ the space
    \begin{equation}\label{eq:homtau}
      \Hom_{\im=\tau}(\cE,\cE')
      :=
      \left\{\cE\xrightarrow{f}\cE'\ |\ \tau(\im f)=\tau\right\}.
    \end{equation}

  \item\label{item:le:taulcl:gr} The space
    \begin{equation*}
      \Hom_{\gr\im=\cF}(\cE,\cE')
      :=
      \left\{\cE\xrightarrow{f}\cE'\ |\ \gr(\im f)\cong \cF\right\}
    \end{equation*}
    for a fixed object $\cF$, with the associated graded $\gr$ as in \Cref{re:genassocgr}.

  \item\label{item:le:taulcl:iso} The space
    \begin{equation*}
      \Hom_{\im=\cF}(\cE,\cE')
      :=
      \left\{\cE\xrightarrow{f}\cE'\ |\ \im f\cong \cF\right\}
    \end{equation*}
    for a fixed bundle $\cF$. Said space is moreover connected. 
\end{enumerate}
\end{proposition}
\begin{proof}
  \begin{enumerate}[(1),wide]
  \item For any fixed rank $r$ the images of $f$ ranging over the locus
    \begin{equation}\label{eq:pr:taulcl:bh}
      \bH_{\rk=r}
      :=
      \Hom_{\rk=r}(\cE,\cE')
      :=
      \left\{\cE\xrightarrow{f}\cE'\ |\ \tau(\im f)=\tau\right\}
    \end{equation}
    (locally closed because definable by (non-)vanishing of minors locally on $E$) aggregate to a rank-$r$ bundle over $\bH_{\rk=r}\times E$, which will in particular be flat over $\bH_{\rk=r}$. This is thus an {\it $\bH_{\rk=r}$-parametrized algebraic family of bundles on $E$} in the sense of \cite[\S 4, p.176]{shatz_dec}, and the local closure is a direct consequence of \cite[discussion following Corollary to Proposition 10, p.183]{shatz_dec}.

  \item Having fixed a type
    \begin{equation}\label{eq:tau}
      \tau:=((r_i,d_i))_{i=1}^s
    \end{equation}
    (a coarser invariant than an associated graded object), there is a morphism
    \begin{equation*}
      \Hom_{\im=\tau}(\cE,\cE')
      \ni f
      \xmapsto{\quad}
      \left(\gr\left(\text{charge-$(r_i,d_i)$ HN-summand of $\im f$}\right)\right)_{i=1}^s
      \in
      \prod_{i=1}^s M(r_i,d_i)
    \end{equation*}
    to the product of {\it coarse moduli spaces} \cite[Theorem 7.2.1]{lepot-vb} $M(r_i,d_i)\cong E^{[\gcd(r_i,d_i)]}$ \cite[Theorem 1]{tu} classifying, respectively, associated graded objects of charge-$(r_i,d_i)$ bundles. The spaces $\Hom_{\gr\im=\cF}$ are nothing but the preimages of that morphism, hence closed in $\Hom_{\im=\tau}$ and locally closed in $\Hom$ by part \Cref{item:le:taulcl:tau}. 
    
  \item This is again a refinement on \Cref{item:le:taulcl:gr} (as that was a refinement on \Cref{item:le:taulcl:tau}): every individual $\Hom_{\gr\im=\bullet}$ is partitioned into $\Hom_{\im=\cF}$, for $\cF$ ranging over the isomorphism classes of bundles with the same associated graded object $\bullet$. For that reason, we assume $\cG:=\gr\cF$ fixed throughout.

    Consider also a stable summand $\cS\le_{\oplus} \cG$ (so one of the finitely many stable subquotients featuring in Jordan-H\"older filtrations of the relevant $\cF$), and write
    \begin{equation*}
      \begin{aligned}
        \sharp_{\cS\bullet n}\cF
        &:=
          \sharp\left\{\text{summands $\tensor*[_h]{\cS}{}$ of $\cF$ with $h\bullet n$}\right\}\\
        &\text{for}\quad
          \bullet\in\{\text{ordering symbols: }=,\ \le,\ >,\ \text{ etc.}\}
      \end{aligned}      
    \end{equation*}
    As a consequence of \Cref{eq:selfexts}, we have
    \begin{equation*}
      \begin{aligned}
        \sharp_{\cS\ge 1}\cF
        &=
          \dim\Hom(\cS,\cF)\\
        \sharp_{\cS=1}\cF+2\sharp_{\cS\ge 2}\cF
        &=
          \dim\Hom(\tensor*[_2]{\cS}{},\cF)\\
        \sharp_{\cS=1}\cF+2\sharp_{\cS = 2}\cF+3\sharp_{\cS\ge 3}\cF
        &=
          \dim\Hom(\tensor*[_3]{\cS}{},\cF)\\
        &\vdots
      \end{aligned}
    \end{equation*}
    The isomorphism class of $\cF$ is thus determined numerically by
    \begin{equation*}
      \dim\Hom(\tensor*[_h]{\cS}{},\cF)
      =
      \dim\Gamma\left(\cF\otimes \tensor*[_h]{\cS}{^*}\right)
    \end{equation*}
    for varying $h$. Those quantities being upper semicontinuous \cite[Theorem III.12.8]{hrt} in $\cF$ ranging over $\Hom_{\gr\im=\cG}$, their joint level sets will be locally closed. 

    As for connectedness, simply observe that $\Hom_{\im=\cF}(\cE,\cE')$ is the image of the morphism
    \begin{equation*}
      \left\{\cE\xrightarrowdbl{\quad\text{epic}\quad}\cF\right\}
      \times
      \left\{\cF\lhook\joinrel\xrightarrow{\quad\text{monic}\quad}\cE'\right\}
      \xrightarrow{\quad\text{composition}\quad}
      \Hom(\cE,\cE')
    \end{equation*}
    with connected domain.
  \end{enumerate}
\end{proof}

\begin{remark}\label{re:strataareschemes}
  \Cref{le:semistabopen} and \Cref{pr:taulcl} (henceforth largely implicit) are what allows dimension estimates of and morphisms defined on spaces of morphisms whose (co)kernel has specific types or is (semi)stable: such spaces are all naturally schemes. 
\end{remark}

There is also an $\Ext$ version of \Cref{pr:taulcl} ($\Ext^1$ as opposed to $\Hom$, that is). We record it here for whatever intrinsic interest it may possess, and also because local closure is achieved by rather different (and less direct) means in a particular case in \cite[Theorem 4.7]{00-leaves_xv3}. Throughout the statement and proof we denote the middle term $\bullet$ of an extension
\begin{equation*}
  \xi
  \quad:\quad
  0\to
  \cE'
  \lhook\joinrel\xrightarrow{\quad}
  \bullet
  \xrightarrowdbl{\quad}
  \cE
  \to 0  
\end{equation*}
corresponding \cite[\S III.5.2]{gm_halg_2e_2003} to a class $\xi\in \Ext^1(\cE,\cE')$ by $\md(\xi)$. 

\begin{proposition}\label{pr:taulcl.ext}
  For bundles $\cE$ and $\cE'$ the following subspaces of $\Ext^1(\cE,\cE')$ are all locally closed.

  \begin{enumerate}[(1),wide]
  \item\label{item:le:taulcl.ext:tau} For any type $\tau$, the space
    \begin{equation}\label{eq:homtau}
      \Ext^1_{\tau}(\cE,\cE')
      :=
      \left\{\xi\in \Ext^1(\cE,\cE')\ |\ \tau(\md(\xi))=\tau\right\}.
    \end{equation}

  \item\label{item:le:taulcl.ext:gr} The space
    \begin{equation*}
      \Ext^1_{\gr=\cF}(\cE,\cE')
      :=
      \left\{\xi\in \Ext^1(\cE,\cE')\ |\ \gr(\md(\xi))\cong \cF\right\}
    \end{equation*}
    for a fixed object $\cF$.

  \item\label{item:le:taulcl.ext:iso} The space
    \begin{equation*}
      \Ext^1_{\cF}(\cE,\cE')
      :=
      \left\{\xi\in \Ext^1(\cE,\cE')\ |\ \md(\xi)\cong \cF\right\}
    \end{equation*}
    for a fixed bundle $\cF$. 
\end{enumerate}
\end{proposition}
\begin{proof}
  The proof is entirely parallel to that of \Cref{pr:taulcl}, the main difference being that the algebraic family of bundles one argues with is parametrized by a different space (than the $\bH_{\rk=r}$ of \Cref{eq:pr:taulcl:bh}). We indicate how that comes about, setting aside the repetitive portion of the argument.
  
  There is a {\it universal extension}
  \begin{equation*}
    \xi_{univ}
    \quad:\quad
    0\to
    \Ext^1(\cE,\cE')^*\otimes \cE'
    \lhook\joinrel\xrightarrow{\quad}
    \cU
    \xrightarrowdbl{\quad}
    \cE
    \to 0    
  \end{equation*}
  defined simply as
  \begin{equation*}
    \xi_{univ}:=\id
    \in
    \End(\Ext^1(\cE,\cE'))
    \cong
    \Ext^1(\cE,\cE')\otimes \Ext^1(\cE,\cE')^*
    \cong
    \Ext^1(\cE,\Ext^1(\cE,\cE')^*\otimes \cE').
  \end{equation*}
  The universality consists in its classifying all extensions of $\cE$ by $\cE'$: $\xi\in \Ext^1(\cE,\cE')$ can be recovered tautologically as the bottom row in
  \begin{equation}\label{eq:pushext}
    \begin{tikzpicture}[>=stealth,auto,baseline=(current  bounding  box.center)]
      \path[anchor=base]
      (0,0) node (l0) {$0$}
      +(2,1) node (lu) {$\Ext^1(\cE,\cE')^*\otimes \cE'$}
      +(2,-1) node (ld) {$\cE'$}      
      +(5,1) node (mu) {$\cU$}
      +(4,-1) node (md) {$\bullet$}
      +(6,0) node (r) {$\cE$}
      +(7,0) node (r0) {$0$}
      +(3.3,0) node (text) {$\text{pushout}$}
      ;
      \draw[->] (l0) to[bend left=6] node[pos=.5,auto] {$\scriptstyle $} (lu);      
      \draw[->] (l0) to[bend right=6] node[pos=.5,auto] {$\scriptstyle $} (ld);
      \draw[right hook->] (lu) to[bend left=6] node[pos=.5,auto] {$\scriptstyle $} (mu);
      \draw[right hook->] (ld) to[bend right=6] node[pos=.5,auto] {$\scriptstyle $} (md);
      \draw[->] (r) to[bend right=0] node[pos=.5,auto] {$\scriptstyle $} (r0);
      \draw[->>] (mu) to[bend left=6] node[pos=.5,auto] {$\scriptstyle $} (r);
      \draw[->>] (md) to[bend right=6] node[pos=.5,auto] {$\scriptstyle $} (r);
      \draw[->] (lu) to[bend right=6] node[pos=.5,auto,swap] {$\scriptstyle \xi\otimes\id_{\cE'}$} (ld);
      \draw[->] (mu) to[bend left=6] node[pos=.5,auto] {$\scriptstyle $} (md);
    \end{tikzpicture}
  \end{equation}
  with $\xi$ along the left-hand vertical arrow regarded as a map $\Ext^1(\cE,\cE')^*\xrightarrow{\xi}\Bbbk$. Restricting attention at no loss to non-zero $\xi$ (hence non-split extensions) only, so that the vertical maps in \Cref{eq:pushext} are both epic,
  \begin{equation*}
    \Ext^1(\cE,\cE')^{\times}:=\Ext^1(\cE,\cE')\setminus\{0\}
    \ni
    \xi
    \xmapsto{\quad}
    \left(\text{right-hand vertical arrow in \Cref{eq:pushext}}\right)
  \end{equation*}
  can be regarded as a map to the {\it Quot scheme} \cite[\S 5.1.4]{MR2223407} $\Quot:=\Quot_{\cU/E}$ parametrizing quotients $\cU\xrightarrowdbl{}\bullet$ identified, as in \cite[\S\S 5.1.2 and 5.1.3]{MR2223407}, if
  \begin{itemize}
  \item their respective kernels are {\it equal};
  \item or equivalently, they are isomorphic as objects in the {\it comma category} \cite[Exercise 3K]{ahs} $\cU\downarrow \Coh(E)$ of arrows $\cU\xrightarrow{}\diamond$ in $\Coh(E)$.
  \end{itemize}
  The argument can then proceed as before, in \Cref{pr:taulcl}: the quotients in $\Quot$ of any given type cut out a locally closed locus by \cite[Proposition 10 and its Corollary]{shatz_dec} again, etc.
\end{proof}

\begin{remark}\label{re:taulcl.ext}
  Given that the projectivizations
  \begin{equation*}
    \bP\Ext_{\cF}^1(\cE,\cE')
    :=
    \Ext_{\cF}^1(\cE,\cE')^{\times}/\Bbbk^{\times}
  \end{equation*}
  feature in \cite[Theorem 1]{FO98} (also \cite[Theorem 1.1]{00-leaves_xv3}, \cite[Theorem A]{ch_sympl-qnk_xv1}) as {\it symplectic leaves} \cite[Definition 4.9]{clm} of a Poisson structure on $\bP\Ext^1(\cE,\cE')$ for $\cE':=\cO$ and a stable $\cE$, their connectedness is crucial. I do not know whether connectedness obtains in the full generality of \Cref{pr:taulcl.ext}, but it certainly does if at least one of $\cE$ and $\cE'$ is stable: if $\cE'$ is, say, the map
  \begin{equation}\label{eq:epi2ext}
    \begin{tikzpicture}[>=stealth,auto,baseline=(current  bounding  box.center)]
      \path[anchor=base] 
      (0,0) node (u) {$\left(\text{non-split }\cF\xrightarrowdbl[\quad]{\pi}\cE\text{ with semistable kernel}\right)$}
      +(0,-2) node (d) {$\left(\text{class of unique extension }0\to \bullet\lhook\joinrel\to \cF\xrightarrowdbl{\pi} \cE\to 0\right)$}
      +(0,-2.5) node (in) {\rotatebox[origin=c]{-90}{$\in$}}
      +(0,-3) node (pext) {$\bP\Ext^1(\cE,\cE')$}
      ;
      \draw[|->] (u) to[bend left=0] node[pos=.5,auto] {$\scriptstyle $} (d);
    \end{tikzpicture}
  \end{equation}
  is well-defined, has stable kernel (for it will have coprime rank and degree $\rk\cE'$ and $\deg\cE'$), and has $\bP\Ext_{\cF}^1(\cE,\cE')$ as its image, for its kernel $\bullet$, assumed semistable, must be isomorphic to $\cE'$: the unique \cite[Proposition 14 and Lemma 30]{tu} (semi)stable bundle of given rank and determinant $\det\cF\otimes (\det\cE)^{-1}$. The domain of \Cref{eq:epi2ext}, if non-empty (or there is nothing to prove), is open dense in $\Hom(\cF,\cE)$ by the open nature of (semi)stability \cite[Theorem 2.8(B)]{maruy_open} and hence connected. 
\end{remark}

\begin{theorem}\label{th:maxnondeg}
  Let $\cE$ and $\cE'$ be non-zero indecomposable bundles on the elliptic curve $E$ with slopes $\mu(\cE)<\mu(\cE')$.
  \begin{enumerate}[(1),wide]
  \item\label{item:th:maxnondeg:diffrk} If $\cE$ and $\cE'$ have unequal ranks, the subspace
    \begin{equation}\label{eq:hommaxinhom}
      \Hom_{max}(\cE,\cE')\subset \Hom(\cE,\cE')\ne \{0\}
    \end{equation}
    consisting of maximally non-degenerate morphisms is open dense.

  \item\label{item:th:maxnondeg:samerk} If $\rk\cE=\rk\cE'$ then $\Hom_{\lhook\joinrel\to}(\cE,\cE')$ is open dense in $\Hom(\cE,\cE')$. 
  \end{enumerate}
\end{theorem}
\begin{proof}
  The full hom space is non-zero from \cite[Lemma 2.5]{ch_sympl-qnk_xv1}, and density will follow from openness (\Cref{le:semistabopen}) as soon as we argue that at least {\it one} morphism is maximally non-degenerate.
  
  Dualizing ($\cE\leftrightarrow \cE^*$ and $\cE'\leftrightarrow \cE'^*$) and changing the roles of $\cE$ and $\cE'$ will turn the two cases $\rk\cE>\rk \cE'$ and $\rk \cE<\rk \cE'$ of \Cref{item:th:maxnondeg:diffrk} into one another, while \Cref{item:th:maxnondeg:samerk} is self-dual (the functor $\bullet^*\cong {\cH}om(\bullet,\cO)$ on the category of vector bundles an injection with torsion cokernel into another such). It thus suffices to assume $\rk\cE\ge\rk \cE'$ ; we moreover focus on item \Cref{item:th:maxnondeg:diffrk}, indicating at the very end of the proof how the argument in fact delivers \Cref{item:th:maxnondeg:samerk} as well. The task, then, is to show that
  \begin{equation}\label{eq:mumurkrk}
    \mu(\cE)<\mu(\cE')
    \quad\text{and}\quad
    \rk(\cE)>\rk(\cE')
    \quad
    \xRightarrow{\quad}
    \quad
    \exists\text{ an epimorphism }\cE\xrightarrowdbl{\quad}\cE'.
  \end{equation}
  In general, morphisms $\cE\to \cE'$ admit epi-mono factorizations
  \begin{equation}\label{eq:comppair}
    \cE
    \xrightarrowdbl{\quad\pi\quad}
    \cF
    \lhook\joinrel\xrightarrow{\quad\iota\quad}
    \cE'
  \end{equation}
  through quotient bundles $\cF$ of $\cE$, which fall into finitely many HN types \cite[\S 5.1, Example]{lepot-vb}.
  
  It is thus enough to argue, having fixed a type $\tau\ne \tau(\cE')$, that the space of morphisms $\cE\to \cE'$ factoring as in \Cref{eq:comppair} through a type-$\tau$ bundle $\cF$ has dimension strictly less than
  \begin{equation}\label{eq:murkrk}
    \deg\left(\cE'\otimes \cE^*\right)
    =
    \mu\left(\cE'\otimes \cE^*\right)\cdot \rk \cE\cdot \rk \cE'
    \xlongequal{\ \text{\cite[Theorem 10.2.1]{lepot-vb}}\ }
    \left(\mu(\cE')- \mu(\cE)\right)\cdot \rk \cE\cdot \rk \cE'.
  \end{equation}
  Each HN summand of $\cF$ will have slope dominating that of $\cE$ and dominated by that of $\cE'$ (because $\cE$ and $\cE'$ are indecomposable and hence \cite[Appendix A, Fact]{tu} semistable). The HN splitting of $\cF$ will be of the form
  \begin{equation}\label{eq:fdec}
    \cF
    \cong
    \cF_{\leftarrow}
    \oplus
    \bigoplus_{1\le i\le s}\cF_{|,i}
    \oplus
    \cF_{\rightarrow},
  \end{equation}
  where
  \begin{itemize}
  \item $\mu(\cE)=\mu(\cF_{\leftarrow})<\mu(\cE')$ and $\cF_{\leftarrow}$ is indecomposable, with the kernels of $\cE\xrightarrowdbl{}\cF_{\leftarrow}$ ranging over finitely many subbundles of $\cE$.
    
  \item Each ``intermediate'' $\cF_{|,i}$ is semistable of slope {\it strictly} buffered by $\mu(\cE)$ and $\mu(\cE')$.

  \item And, at the other extreme from $\cF_{\leftarrow}$, $\cF_{\rightarrow}$ is indecomposable of slope $\mu(\cF_{\rightarrow})=\mu(\cE')$, and ranges over finitely many subbundles of $\cE'$. 
  \end{itemize}
  The trichotomy
  \begin{equation*}
    \mu(\cE)
    \quad\text{vs.}\quad
    \text{ open interval }(\mu(\cE),\mu(\cE'))
    \quad\text{vs.}\quad
    \mu(\cE')
  \end{equation*}
  is valid for the slopes of the HN summands in any case, so only the indecomposability claims require justification. To that end, consider $\cF_{\leftarrow}$ (the discussion of $\cF_{\rightarrow}$ being entirely parallel). The epimorphism $\cE\xrightarrowdbl{}\cF_{\leftarrow}$ lies in the abelian category \cite[Proposition 5.3.6]{lepot-vb} of semistable bundles of slope $\mu:=\mu(\cE)=\mu(\cF_{\rightarrow})$. $\cE$ being indecomposable, it is an iterated self-extension of a stable bundle $\cE_0$ of charge
  \begin{equation*}
    (r,d),\quad \frac dr=\mu(\cE)\text{ in lowest terms}
  \end{equation*}
  (see \Cref{re:selfext}). Furthermore, the length of $\cE$ (i.e. the number of simple subquotients in a {\it Jordan-H\"older filtration} \cite[p.76]{lepot-vb}, all isomorphic to $\cE_0$) is $h_{\cE}:=\frac{\rk \cE}{r}=\frac{\deg\cE}{d}$. The same goes for $\cF_{\rightarrow}$, which is an iterated self-extension of the same $\cE_0$ of length only
  \begin{equation*}
    h_{\cF_{\rightarrow}}:=\frac{\rk \cF_{\rightarrow}}{r}=\frac{\deg\cF_{\rightarrow}}{d} \le h_{\cE}.
  \end{equation*}
  Consider, now, the contributions the three types of summands make to the (dimension of the) space of morphisms $\cE\to \cE'$ factorizable through \Cref{eq:fdec}. 
  
  \begin{enumerate}[(I),wide]
  \item\label{item:th:maxnondeg:mid}{\bf : $\left(\cF_{|,i}\right)$} The slopes of the HN summands of $\cF_{|,i}$ are ordered {\it strictly} between those of $\cE$ and $\cE'$, so the relations
    \begin{equation*}
      \begin{aligned}
        \dim\Hom(\cE,\cF_{|})&=\deg(\cF_{|}\otimes \cE^*)
                               =\left(\mu(\cF_{|})- \mu(\cE)\right)\cdot \rk \cE\cdot \rk \cF_{|}
                               \quad\text{and}\\
        \dim\Hom(\cF_{|},\cE')&=\deg(\cE'\otimes \cF_{|}^*)
                                =\left(\mu(\cE')- \mu(\cF_{|})\right)\cdot \rk \cF_{|}\cdot \rk \cE'
                                \quad\text{\Cref{eq:degrk}}
      \end{aligned}
    \end{equation*}
    still hold for
    \begin{equation}\label{eq:fbar}
      \cF_{|}:=\bigoplus_i \cF_{|,i}. 
    \end{equation}
    That summand thus contributes 
    \begin{equation}\label{eq:sumdiff}
      \left(\mu(\cF_{|})- \mu(\cE)\right)\cdot \rk \cE\cdot \rk \cF_{|}
      +
      \left(\mu(\cE')- \mu(\cF_{|})\right)\cdot \rk \cF_{|}\cdot \rk \cE'
    \end{equation}
    to the dimension of the space of composable pairs \Cref{eq:comppair}.

    By \cite[Theorem 10]{Atiyah} again, each of the $s$ (say) indecomposable summands of the possible $\cF_{|}$ of fixed type is parametrized by $E$, adding at most $s$ to \Cref{eq:sumdiff} as $\cF_{|}$ ranges over the given type. On the other hand, $\Aut(\cF_{|})$ operates on the space of composable pairs \Cref{eq:comppair} by
    \begin{equation*}
      (\pi,\iota)
      \xmapsto[\alpha\in \Aut(\cF_{|,i})]{\quad\alpha\triangleright\quad}
      (\alpha\circ\pi,\iota\circ\alpha^{-1})
    \end{equation*}
    without altering the composition; this {\it subtracts} (at least) $s$ from the target dimension, since $\Aut(\cF_{|})$ contains at least the automorphisms scaling each of the $s$ summands. This gives \Cref{eq:sumdiff} as an upper bound for the contribution of the middle summand \Cref{eq:fbar} of \Cref{eq:fdec} to the target dimension.
    

    
  \item\label{item:pr:maxnondeg:=<} {\bf : $\left(\cF_{\leftarrow}\right)$} The space of compositions \Cref{eq:comppair} through $\cF_{\leftarrow}$ has dimension
    \begin{equation}\label{eq:flcontrib}
      \begin{aligned}
        \dim\Hom(\cF_{\leftarrow},\cE')
        =
        \deg(\cE'\otimes \cF_{\leftarrow}^*)
        &=
          \left(\mu(\cE')- \mu(\cF_{\leftarrow})\right)\cdot \rk \cF_{\leftarrow}\cdot\rk \cE'\\
        &=
          \left(\mu(\cE')- \mu(\cE)\right)\cdot \rk \cF_{\leftarrow}\cdot\rk \cE'. 
      \end{aligned}      
    \end{equation}
    
    
  \item\label{item:pr:maxnondeg:<=} {\bf : $\left(\cF_{\rightarrow}\right)$} This is not materially different from \Cref{item:pr:maxnondeg:=<}, so we omit the argument (which consists of reversing the roles of $\cE$ and $\cE'$ in the preceding discussion, working with embeddings in place of surjections, etc.). The analogue of \Cref{eq:flcontrib} is
    \begin{equation}\label{eq:frcontrib}
      \begin{aligned}
        \dim\Hom(\cE,\cF_{\rightarrow})
        =
        \deg(\cF_{\rightarrow}\otimes \cE^*)
        &=
          \left(\mu(\cF_{\rightarrow})- \mu(\cE)\right)\cdot \rk \cE\cdot \rk \cF_{\rightarrow}\\
        &=
          \left(\mu(\cE')- \mu(\cE)\right)\cdot \rk \cE\cdot \rk \cF_{\rightarrow}. 
      \end{aligned}      
    \end{equation}
  \end{enumerate}
  To sum it all up, we will be done once we have argued that unless $\cF=\cE'$,
  \begin{equation}\label{eq:keyineq}
    \text{\Cref{eq:sumdiff}}
    +
    \text{\Cref{eq:flcontrib}}
    +
    \text{\Cref{eq:frcontrib}}
    <
    \text{\Cref{eq:murkrk}};
  \end{equation}
  note here that we have fixed specific $\cF_{\leftrightarrow}$, for, as indicated, there are finitely many choices for each.
  
  Because
  \begin{equation*}
    \left(\mu(\bullet)-\mu(\square)\right)\cdot \rk\bullet\cdot \rk\square
  \end{equation*}
  is bilinear (i.e. additive in each variable, keeping the other fixed), the left-hand side of \Cref{eq:keyineq} is nothing but 
  \begin{equation*}
    \left(\mu(\cF)- \mu(\cE)\right)\cdot \rk \cE\cdot \rk \cF
    +
    \left(\mu(\cE')- \mu(\cF)\right)\cdot \rk \cF\cdot \rk \cE'.
  \end{equation*}
  Subtracting this quantity from
  \begin{equation*}
    \text{\Cref{eq:murkrk}}
    =
    \left(\mu(\cF)- \mu(\cE)\right)\cdot \rk \cE\cdot \rk \cE'
    +
    \left(\mu(\cE')- \mu(\cF)\right)\cdot \rk \cE\cdot \rk \cE'
  \end{equation*}
  produces
  \begin{equation}\label{eq:2nonnegterms}
    \left(\mu(\cF)- \mu(\cE)\right)\cdot \rk \cE\cdot (\rk \cE'-\rk \cF)
    +
    \left(\mu(\cE')- \mu(\cF)\right)\cdot (\rk \cE-\rk \cF)\cdot \rk \cE',
  \end{equation}
  a sum with non-negative terms because $\mu(\cE)\le \mu(\cF)\le \mu(\cE')$ (with at least one inequality strict, since $\mu(\cE)<\mu(\cE')$) and $\rk(\cF)\le \rk(\cE')<\rk(\cE)$. The only way to have {\it both} terms of \Cref{eq:2nonnegterms} vanish is thus
  \begin{equation*}
    \mu(\cF)=\mu(\cE')>\mu(\cE)
    \xRightarrow{\quad}
    \rk\cF=\rk \cE',
  \end{equation*}
  so that $\cF\cong \cE'$ (the ranks and slopes, hence also degrees, being equal).

  To conclude, observe that the argument holds for $\rk\cE=\rk\cE'$ as well, up to the very end: the vanishing of (both terms of) \Cref{eq:2nonnegterms} implies
  \begin{equation*}
    \rk\cF<\rk\cE=\rk\cE'
    \xRightarrow{\quad}
    \mu(\cE)=\mu(\cF)=\mu(\cE'),    
  \end{equation*}
  contradicting $\mu(\cE)<\mu(\cE')$ and hence proving that $\rk\cF=\rk\cE=\rk\cE'$ (all that claim \Cref{item:th:maxnondeg:samerk} requires).   
\end{proof}

\begin{remarks}\label{res:pr:maxnondeg:post}
  \begin{enumerate}[(1),wide]
    
  \item\label{item:res:pr:maxnondeg:post:recpolish} Cf. \cite[Corollary 14.11]{Polishchuk-book}, recoverable from \Cref{th:maxnondeg} as a special case: for stable (hence indecomposable \cite[Lemma 12]{tu}) $\cE$ and $\cE'$ with $\mu(\cE)<\mu(\cE')$, if $\dim\Hom(\cE,\cE')=1$ then any non-zero morphism $\cE\to \cE'$ is epic or monic, depending on whether $\rk\cE>\rk\cE'$ or $\rk\cE<\rk\cE'$ respectively. 
    
  \item\label{item:res:pr:maxnondeg:post:strictineq} Both strict inequality requirements in \Cref{th:maxnondeg}\Cref{item:th:maxnondeg:diffrk} are necessary, as neither is sufficient if the other is relaxed to (possible) equality:
    \begin{enumerate}[{},wide]
    \item {\bf Unequal ranks.} A degree-1 line bundle $\cO(z)$, $z\in E$ has global sections, but there are no {\it epi}morphisms $\cO\to \cO(z)$ (even though the slopes are strictly ordered). This makes it clear why the modification in \Cref{item:th:maxnondeg:samerk}, allowing for torsion in the cokernel, was necessary.
            
    \item {\bf Unequal slopes.} As recalled in \cite[Remark 2.6(2)]{ch_sympl-qnk_xv1}, degree-0 line bundles have no non-zero sections, so in particular cannot be surjected upon by trivial bundles (no matter how high the rank).
      
      By induction on $k$, they also cannot be surjected upon by the unique \cite[Theorem 5]{Atiyah} indecomposable rank-$k$ degree-0 vector bundle with non-zero section space ($\tensor*[_k]{\cO}{}$ in the notation of \Cref{re:selfext}), so even indecomposability will not make up the deficiency.
    \end{enumerate}
  \item\label{item:res:pr:maxnondeg:post:indec} Indecomposability is also essential to \Cref{th:maxnondeg}, on both sides (see \Cref{cor:th:maxnondeg:dec} below though, on how it can be weakened appropriately):
    \begin{enumerate}[{},wide]
    \item {\bf Large-rank bundle.} Or: in proving claim \Cref{eq:mumurkrk}, we need the domain $\cE$ to be indecomposable. Specifically, indecomposability was used crucially in analyzing the summand $\cF_{\leftarrow}$ of \Cref{eq:fdec}. Indeed, there are no epimorphisms
      \begin{equation}\label{eq:item:res:pr:maxnondeg:post:indec:o2op}
        \cO^k
        =:
        \cE
        \xrightarrow{\quad}
        \cE'
        :=
        \cO(p)
        ,\quad\text{arbitrary } k\in \bZ_{\ge 0}
        \quad\text{and}\quad
        p\in E.
      \end{equation}
      The more general phenomenon at work here is the fact that {\it no} semistable slope-1 bundle $\cE'$ can be generated by its global sections: the canonical map
      \begin{equation*}
        \Gamma(\cE')\otimes \cO
        \overset{\text{\Cref{eq:degrk}}}{\cong}
        \cO^{\rk \cE'}
        \xrightarrow{\quad}
        \cE'
      \end{equation*}
      is an embedding with torsion cokernel, and morphisms into $\cE'$ from {\it any} $\cO^k$ must factor through its image.
      
      Note also in passing that this argument (which applies all the more to bundles of slope {\it less} than 1) also provides a converse to \cite[Lemma 4.8(4)]{CKS3}, to the effect that semistable bundles of slope $>1$ {\it are} generated by global sections. Coupling it with \cite[Lemma 4.8(3) and (4)]{CKS3} yields the conclusion that a semistable bundle on $E$ is generated by global sections precisely when its slope is $>1$. 
      
    \item {\bf Small-rank bundle.} The claim now is that \Cref{eq:mumurkrk} (in its reliance on the indecomposability of the summand $\cF_{\rightarrow}$ of \Cref{eq:fdec}) also requires that $\cE'$ be indecomposable, even if we already assume $\cE$ is. Take, say,
      \begin{equation}\label{eq:item:res:pr:maxnondeg:post:indec:e2l2}
        \begin{aligned}
          \cE\text{ indecomposable with }\zeta(\cE)=(k,n)=(k,kd-1),\quad k>2\\
          \cE':=\cL^{\oplus 2}
          ,\quad
          \cL\text{ line bundle of degree }d.
        \end{aligned}        
      \end{equation}
      By \Cref{eq:degrk} (and \Cref{eq:brake2f}),
      \begin{equation*}
        \dim\Hom(\cE,\cL) = \braket{\cE\to \cL} = kd-(kd-1)=1
      \end{equation*}
      and hence
      \begin{equation*}
        \dim\Hom(\cE,\cE')
        =
        \dim\Hom(\cE,\cL^{\oplus 2})
        =
        2\dim\Hom(\cE,\cL) = 2.
      \end{equation*}
      Because all morphisms $\cE\to \cE'=\cL^{\oplus 2}$ must factor through one of the embeddings $\cL\subset \cL^{\oplus 2}$, there are no {\it epi}morphisms $\cE\xrightarrowdbl{} \cE'$.

      The same $\cE$ and $\cL$ illustrate the relevance of indecomposability to \Cref{th:maxnondeg}\Cref{item:th:maxnondeg:samerk} as well: morphisms $\cE\to \cE':=\cL^{\oplus k}$ (between equal-rank bundles, with the domain indecomposable) cannot be injective. 
    \end{enumerate}

    
  \item\label{item:res:pr:maxnondeg:post:fixindec} Item \Cref{item:res:pr:maxnondeg:post:indec} suggests the possible modes of failure for decomposable (though still semistable) bundles: in both examples the issue is that one of $\cF_{\leftrightarrow}$ ranges over a family, counterbalancing the dimension gap.
    \begin{itemize}[wide]
    \item In the first case, of morphisms $\cO^{k}\to \cO(p)$, the possibilities for $\cF_{\leftarrow}$ range over quotients $\cO^k\xrightarrowdbl{}\cO$ and the dimension of the space of such quotients is precisely $k=\dim\Hom(\cO^k\to \cO(p))$.

    \item Similarly, in the other case the $\cF_{\rightarrow}\lhook\joinrel\to \cL^2$ are the embeddings of the various copies of $\cL$; again, the space of such embeddings is $\left(2=\dim\Hom(\cE,\cL^2)\right)$-dimensional. 
    \end{itemize}
    The proof of \Cref{th:maxnondeg} will still go through just fine, though, if $\cE$ and $\cE'$ are semistable, each with only finitely many subbundles of equal slope. This is equivalent to their being {\it maximally asymmetric}, in the sense of having minimal-dimensional automorphism group given the charge $(k,n)$. Suppose
    \begin{equation*}
      \cE\cong \bigoplus_i\cE_i
      ,\quad
      \cE_i\text{ indecomposable}
      ,\quad
      \mu(\cE_i)=\mu:=\frac kn.
    \end{equation*}
    Per \Cref{re:selfext}, each $\cE_i$ is an $h_i$-fold self-extension of a single slope-$\mu$ stable bundle for $h_i:=\gcd(\rk\cE_i,\ \deg\cE_i)$. The minimal dimension achievable by $\Aut(\cE)$ is then easily seen to be the total length
    \begin{equation}\label{eq:lengthe}
      \mathrm{length}(\cE)
      =
      \bigoplus_i
      \mathrm{length}(\cE_i)
    \end{equation}
    in the abelian category \cite[Proposition 5.3.6]{lepot-vb} of slope-$\mu$ semistable bundles, and that equality is achieved precisely when the associated gradings of the $\cE_i$ have no common stable summands. In other words: $\cE_i$ is an $h_i$-fold self-extension of the stable $\cE_i'$, with all $\cE_i'$ mutually non-isomorphic. We record the resulting improvement on \Cref{th:maxnondeg} as \Cref{cor:th:maxnondeg:dec}.

  \item\label{item:res:pr:maxnondeg:post:basic} The maximally asymmetric slope-$\mu$ semistable bundles can also be characterized alternatively (assuming semistability and the fixed slope) as precisely those $\cE$ whose endomorphism algebra $\End(\cE)$ is {\it basic} \cite[pp.172-173]{sy_frob-1}. As we are working over an algebraically closed field, this is equivalent \cite[Proposition II.6.19]{sy_frob-1} to the simple $\End(\cE)$-modules being 1-dimensional. For this reason, we occasionally (and for brevity and variety) refer to maximally asymmetric bundles as {\it basic}.
  \end{enumerate}
\end{remarks}

\begin{corollary}\label{cor:th:maxnondeg:dec}
  The conclusion of \Cref{th:maxnondeg} holds also for basic semistable $\cE$ and $\cE'$.  \qedhere
\end{corollary}

As a bit of a side note, the example in \Cref{eq:item:res:pr:maxnondeg:post:indec:o2op} also extends in a different direction (beyond what \Cref{res:pr:maxnondeg:post}\Cref{item:res:pr:maxnondeg:post:indec} already observes): coproducts of arbitrary stable $\cS$ in place of $\cO^k$. The construction is based on the following remark. We write $\cO_p$ for the {\it skyscraper sheaf} \cite[Exercise II.1.17]{hrt} attached to a point $p\in E$. 

\begin{lemma}\label{le:sttwist}
  Let $\cS$ be a stable bundle and $p\in E$. The middle term attached to the universal extension
  \begin{equation}\label{eq:sbop}
    0\to
    \Ext^1(\cO_p,\cS)^*\otimes \cS
    \xrightarrow{\quad}
    \bullet
    \xrightarrow{\quad}
    \cO_p
    \to 0
  \end{equation}
  is indecomposable and hence also stable, of charge $(\rk^2\cS,\ \rk\cS\cdot \deg\cS+1)$. 
\end{lemma}
\begin{proof}
  The claim about the charge follows from
  \begin{equation*}
    \Ext^1(\cO_p,\cS)^*
    \cong
    \Hom(\cS,\cO_p)
    \quad\left(\text{Serre duality}\right),
  \end{equation*}
  a space of dimension $\rk\cS$. Semistable bundles of that charge must indeed be stable \cite[Corollary 14.8]{Polishchuk-book}, for the rank and degree components of the charge are coprime. 
  
  As for indecomposability, simply note that $\bullet$ is by its very definition the image of the indecomposable sheaf $\cO_p$ through the modified {\it twist functor} denoted by $T'_{\cS}$ in \cite[Definition 2.7]{st}. That endofunctor of the bounded derived category $D^b(E)$ of coherent sheaves on $E$ is an equivalence by \cite[Proposition 2.10]{st}: $\cS$ is stable and hence {\it (1-)spherical} in the sense of \cite[Definition 2.9]{st} (or \cite[Definition 8.1]{Huy-FM}). Indeed, skyscraper sheaves on curves are 1-spherical \cite[Example 8.10(i)]{Huy-FM}, and every stable bundle is mapped onto one such by some autoequivalence of $D^b(E)$ \cite[Theorem 14.7]{Polishchuk-book}. 
\end{proof}

The aforementioned expansion of \Cref{eq:item:res:pr:maxnondeg:post:indec:o2op} is now as follows. 

\begin{example}\label{ex:sttwist}
  Fix an arbitrary stable $\cS$ and a point $p\in E$ and let $\cE'$ be the (also stable, by \Cref{le:sttwist}) middle term $\bullet$ of \Cref{le:sttwist}. Morphisms $\cS\to \cE'$ all factor through the left-hand coproduct $\cS^{\rk\cS}\cong \Ext^1(\cO_p,\cS)^*\otimes \cS$ of \Cref{eq:sbop}, as evident from the fragment
  \begin{equation*}
    \begin{tikzpicture}[>=stealth,auto,baseline=(current  bounding  box.center)]
      \path[anchor=base] 
      (0,0) node (l0) {$0$}
      +(2,0) node (ll) {$\Hom(\cS,\cS^{\rk \cS})$}
      +(5,1) node (l) {$\Hom(\cS,\cE')$}
      +(8,1) node (r) {$\Hom(\cS,\cO_p)$}
      +(4,-2) node (rr) {$\Ext^1(\cS,\cS^{\rk \cS})$}
      +(10,0) node (r0) {$\Ext^1(\cS,\cE')\overset{\text{\Cref{eq:degrk}}}{\cong} 0$}
      ;
      \draw[->] (l0) to[bend right=6] node[pos=.5,auto] {$\scriptstyle $} (ll);
      \draw[right hook->] (ll) to[bend left=6] node[pos=.5,auto] {$\scriptstyle $} (l);
      \draw[->] (l) to[bend left=6] node[pos=.5,auto] {$\scriptstyle 0$} (r);
      \draw[-implies,double equal sign distance] (6.5,.3) to (6.5,1);
      \draw[->>] (r.east) .. controls +(3,0) and +(-3,0) .. node[pos=.5,auto,sloped] {$\scriptstyle \text{epimorphism of equidimensional spaces}$} node[pos=.5,auto,sloped,swap] {$\scriptstyle \text{hence $\cong$}$} (rr.west);
      \draw[->] (rr) to[bend right=6] node[pos=.5,auto] {$\scriptstyle $} (r0);
    \end{tikzpicture}
  \end{equation*}
  of the long exact cohomology sequence resulting from applying $\Hom(\cS,-)$ to \Cref{eq:sbop}. Consequently, there are no epic $\cE:=\cS^k\to \cE'$.
\end{example}

It is also perhaps worth highlighting which class of morphisms the proof of \Cref{th:maxnondeg} singles out as crucial in assessing the size of $\Hom_{max}(\cE,\cE')$. 

\begin{theorem}\label{th:muextr}
  Let $\cE$ and $\cE'$ be semistable bundles with $\mu(\cE)<\mu(\cE')$.

  One of the locally closed spaces $\Hom_{\im=\tau}(\cE,\cE')$ of \Cref{eq:homtau}, with either 
    \begin{equation}\label{eq:2muextr}
      \mu_{\min}(\tau)=\mu(\cE)
      \quad\text{or}\quad
      \mu_{\max}(\tau)=\mu(\cE'),
    \end{equation}
    has full dimension in $\Hom(\cE,\cE')$ and hence contains an open dense subset thereof. 
\end{theorem}
\begin{proof}
  This is already covered by (the proof of) \Cref{th:maxnondeg}: the complement
  \begin{equation*}
    \Hom(\cE,\cE')
    \setminus
    \coprod_{\tau}
    \Hom_{\im=\tau}(\cE,\cE')
    \quad
    \left(\text{disjoint union, finite by \cite[\S 5.1, Example]{lepot-vb}}\right)
  \end{equation*}
  consists of morphisms factoring through \Cref{eq:fdec} with no extreme terms $\cF_{\leftrightarrow}$. The dimension count under heading \Cref{item:th:maxnondeg:mid} of the proof goes through to prove that that complement has dimension $<\dim\Hom(\cE,\cE')$. 
\end{proof}

And an immediate consequence:

\begin{corollary}\label{cor:homlrarrow}
  For semistable $\cE$ and $\cE'$ with $\mu(\cE)<\mu(\cE')$ the subspaces
  \begin{equation*}
    \begin{aligned}
      \Hom_{\leftrightarrow:extr}(\cE,\cE')
      &:=
        \Hom_{\leftarrow:extr}(\cE,\cE')\cup \Hom_{\rightarrow:extr}(\cE,\cE'),\\
      \Hom_{\leftarrow:extr}(\cE,\cE')
      &:=
        \left\{\cE\xrightarrow{f}\cE'\ |\ \mu_{min}(\im f)=\mu(\cE)\right\}
        \quad\text{and}\\
      \Hom_{\rightarrow:extr}(\cE,\cE')
      &:=
        \left\{\cE\xrightarrow{f}\cE'\ |\ \mu_{max}(\im f)=\mu(\cE')\right\}
    \end{aligned}
  \end{equation*}
  of $\Hom(\cE,\cE')$ are locally closed, and at least one of the latter two has complement of dimension $<\dim \Hom(\cE,\cE')$ and hence contains a dense open subset of $\Hom(\cE,\cE')$. 

  In particular, $\Hom_{\leftrightarrow:extr}(\cE,\cE')$ contains a dense open subset of $\Hom(\cE,\cE')$.  \qedhere
\end{corollary}

\begin{notation}\label{not:tauextr}
  In the context of \Cref{th:muextr}, we write $\tau_{extr}:=\tau_{extr}(\cE\to \cE')$ (as in {\it extreme}) for the unique type $\tau$ with
  \begin{equation*}
    \Hom_{\im=\tau}(\cE,\cE')
    \subset
    \Hom(\cE,\cE')
  \end{equation*}
  full-dimensional.
\end{notation}

As evident from \Cref{res:pr:maxnondeg:post}\Cref{item:res:pr:maxnondeg:post:indec}, the relative size of the ranks is not sufficient to determine which of the possibilities in \Cref{eq:2muextr} obtains (if not both): in both cases $\rk\cE>\rk\cE'$, and each of the two options occurs in one instance. We do, however, have the following simple observation. 

\begin{corollary}\label{cor:whichextr}
  Let $\cE$ and $\cE'$ be semistable bundles with $\mu(\cE)<\mu(\cE')$.

  If $\cE$ ($\cE'$) is maximally asymmetric in the sense of \Cref{res:pr:maxnondeg:post}\Cref{item:res:pr:maxnondeg:post:fixindec} then either
  \begin{itemize}[wide]
  \item $\Hom_{max}(\cE,\cE')$ (when $\rk\cE\ne \rk\cE'$) or $\Hom_{\lhook\joinrel\to}(\cE,\cE')$ (when $\rk\cE=\rk\cE'$) is non-empty (equivalently, open dense), in which case the type $\tau_{extr}$ of \Cref{not:tauextr} is semistable of the same slope as the bundle of smaller rank between $\cE$ and $\cE'$;

  \item or
    \begin{equation*}
      \mu_{max}(\tau_{extr}) = \mu(\cE')
      \quad
      \left(\mu_{min}(\tau_{extr})=\mu(\cE)\text{ respectively}\right).
    \end{equation*}
  \end{itemize}  
\end{corollary}
\begin{proof}
  Once more, this is effectively verified in the course of proving \Cref{th:maxnondeg}, under labels \Cref{item:pr:maxnondeg:=<} and/or \Cref{item:pr:maxnondeg:<=}: if $\cE$ (say) is maximally asymmetric then it has only finitely many slope-$\mu(\cE)$ quotient bundles and subbundles, so the space of morphisms $\cE\to \cE'$ factoring through \Cref{eq:fdec} with no right-hand extreme term $\cF_{\rightarrow}$ is strictly less than full-dimensional. 
\end{proof}

The following observation answers the natural follow-up question to \Cref{le:semistabopen}\Cref{item:le:semistabopen:ss} in one particular case. Recalling from \Cref{res:pr:maxnondeg:post}\Cref{item:res:pr:maxnondeg:post:fixindec} that basic semistable bundles are those whose automorphism group has minimal dimension ($=$ length of the bundle in the abelian category of semistable bundles of the given slope), we have:

\begin{proposition}\label{pr:kcofstab}
  Let $\cE$ and $\cE'$ be unequal-rank semistable bundles, with the lower-rank stable, the higher-rank basic, and $\mu(\cE)<\mu(\cE')$. The locus $\Hom_{max\mid ss}(\cE,\cE')$ of \Cref{eq:hommaxss} is then open dense in $\Hom(\cE,\cE')$. 
\end{proposition}
\begin{proof}
  We assume $\rk\cE<\rk\cE'$ to fix ideas. Given \Cref{cor:th:maxnondeg:dec} and \Cref{le:semistabopen}\Cref{item:le:semistabopen:ss}, the claim reduces to showing that the cokernel of at least {\it one} embedding $\cE\lhook\joinrel\to\cE'$ is semistable. This will be a dimension count, very much along the same lines as the proof of \cite[Theorem 2.4, backward implication, heading (II)]{ch_sympl-qnk_xv1}.

  The space of embeddings $\cE\lhook\joinrel\to \cE'$ modulo $\Aut(\cE)\cong \Bbbk^{\times}$ (this is where {\it stability} is used) has dimension $\braket{\cE\to \cE'}-1$, and each such embedding arises as the kernel of some epimorphism $\cE'\xrightarrowdbl{}\cQ$. We argue that the space of such kernels (i.e. epimorphisms modulo $\Aut(\cQ)$) for $\cQ$ ranging over unstable types
  \begin{equation}\label{eq:type}
    ((k_i,n_i))_{i=1}^s
    \quad\left(\text{unstable meaning }s\ge 2\right)
  \end{equation}
  has strictly smaller dimension.

  There are finitely many types to consider (the family of quotients of a given bundle is {\it bounded}: \cite[\S 5.1, Example on p.72]{lepot-vb}), so it suffices to consider a single type \Cref{eq:type} fixed throughout the rest of the proof (the embeddings $\cE\lhook\joinrel\to\cE'$ yielding that type is then a scheme as in \Cref{re:strataareschemes}).

  The parameter space for the $s$ HN summands of
  \begin{equation*}
    \cQ\cong \bigoplus_{i=1}^s \cQ_i
    ,\quad
    \cQ_i\text{ semistable}
    ,\quad
    \mu(\cQ_i)\text{ decreasing}
  \end{equation*}
  is $(s-1)$-dimensional: each indecomposable bundle of fixed degree and rank is parametrized by $E$ \cite[Theorem 10]{Atiyah}, and there is a total determinant constraint
  \begin{equation*}
    \prod_i \det\cQ_i
    \cong
    \det\cQ
    \cong \det\cE'\otimes (\det\cE)^{-1}. 
  \end{equation*}
  On the other hand, $\Aut(\cQ)$ is at least $(s+1)$-dimensional: scaling of each summand and non-zero morphisms from smaller to (strictly) larger slopes. In summary, the dimension (modulo scaling) of the space of embeddings $\cE\lhook\joinrel\cE'$ resulting in quotients of type \Cref{eq:type} is
  \begin{equation*}
    \le
    \braket{\cE'\to \cQ} + (s-1) - (s+1)
    =
    \braket{\cE'\to \cQ} -2
    =
    \braket{\cE\to \cE'}-2,
  \end{equation*}
  the last equality (a simple computation) being valid whenever there is a short exact sequence $0\to \cE\to \cE'\to \cQ\to 0$. This is indeed smaller than
  \begin{equation*}
    \dim\Hom_{max}(\cE,\cE')/\Bbbk^{\times}
    =
    \braket{\cE\to\cE'}-1,
  \end{equation*}
  concluding. 
\end{proof}

Along the same circle of ideas as \cite[Theorem 2.4]{ch_sympl-qnk_xv1} and \Cref{th:maxnondeg}, we have the following consequent variant on extensions consisting of bundles of prescribed isomorphism class.

\begin{corollary}\label{cor:stabmidstab}
  Let $\cK$, $\cE$ and $\cF$ be bundles, with slopes, ranks and determinants satisfying \Cref{eq:slopesrksdets} (so that $\zeta(\cE)=\zeta(\cK)+\zeta(\cF)$).
  
  If $\cK$ and $\cF$ are stable and $\cE$ basic then
  \begin{itemize}[wide]
  \item the space of embeddings $\cK\lhook\joinrel\to \cE$ fitting into an exact sequence \Cref{eq:kef}.
    \begin{equation*}
      0\to \cK\lhook\joinrel\xrightarrow{\quad} \cE\xrightarrowdbl{\quad} \cF\to 0
    \end{equation*}
    is open dense in $\Hom(\cK,\cE)$;

  \item and dually, the space of epimorphisms $\cE\xrightarrowdbl{}\cF$ that can be completed to extensions \Cref{eq:kef} is open dense in $\Hom(\cE,\cF)$. 
  \end{itemize}
  In particular, extensions \Cref{eq:kef} exist.
\end{corollary}
\begin{proof}
  The two bulleted claims are mutual duals (as stated), and the last clearly follows. As for the first claim, it is a particular instance of \Cref{pr:kcofstab} (applied to $\cK\lhook\joinrel\to \cE$ in place of $\cE\lhook\joinrel\to \cE'$): the stability of the quotient $\cE/\cK$ identifies it with $\cF$, for stable bundles are uniquely determined by their ranks and determinants \cite[Corollary to Theorem 7]{Atiyah}.
\end{proof}


\begin{remark}\label{re:polish.th.14.10}
  As a particular case of \Cref{cor:stabmidstab} we recover the fact \cite[Theorem 14.10]{Polishchuk-book} that whenever $\dim\Ext^1(\cF,\cK)=1$ (for stable $\cF$ and $\cK$), so that all middle terms $\cE$ fitting into non-split extensions \Cref{eq:kef} are isomorphic, those middle terms are stable: we know from \Cref{cor:stabmidstab} that all indecomposable $\cE$ of determinant $\det\cK\otimes \det\cF$ feature in such exact sequences, and there are precisely $\gcd(\rk\cE,\deg\cE)$ isomorphism classes of such indecomposable bundles \cite[Theorem 10]{Atiyah}.
\end{remark}

The proof of \Cref{pr:kcofstab} can be leveraged into more than that statement.

\begin{theorem}\label{th:asym}
  Let $\cE$ and $\cE'$ be unequal-rank basic semistable bundles with $\mu(\cE)<\mu(\cE')$. The subspace
  \begin{equation*}
    \Hom_{max\mid ss\mid asym}(\cE,\cE')
    :=
    \left\{f\in \Hom_{max}(\cE,\cE')\ |\ \cat{kc}(f)\text{ semistable basic}\right\}
  \end{equation*}
  of $\Hom(\cE,\cE')$ is open dense. 
\end{theorem}
\begin{proof}
  We assume once more that $\rk\cE<\rk\cE'$ for definiteness, so that $\cat{kc}=\mathrm{coker}$, and abbreviate $\Hom_{\bullet}(\cE,\cE')$ to $\bH_{\bullet}$ (as in \Cref{eq:bh}). 

  \begin{enumerate}[(I),wide]
  \item\label{item:th:asym:stab} {\bf Stable small-rank bundle.} The issue is openness; that settled, the dimension count in the proof of \Cref{pr:kcofstab} in fact yields the claim: cokernels
    \begin{equation*}
      \cQ\cong \bigoplus_{i=1}^s\cQ_i
      ,\quad
      \cQ_i\text{ indecomposable}
      ,\quad
      \mu(\cQ_i)\text{ all equal}
    \end{equation*}
    of morphisms $f\in \bH_{max\mid ss}$ are parametrized by an $(s-1)$-dimensional scheme (again, one copy of $E$ parametrizing each $\cQ_i$ and a determinant constraint). When $f$ lies outside $\bH_{max\mid ss\mid asym}$ the automorphism group of $\cQ$ has dimension $\ge s+1$, so the previous dimension count does indeed go through. 

    To address openness, note that the cokernel of the universal morphism
    \begin{equation*}
      \cE_{\bH_{max}}
      \xrightarrow{\quad}
      \cE'_{\bH_{max}}
      ,\quad
      \cE_{\bullet}:=\text{pullback to }\bullet\times \cE
      \quad\text{(cf. \Cref{eq:univmor})}
    \end{equation*}
    is a rank-$(\rk\cE'-\rk\cE)$ bundle $\cQ_{\bH_{max}}$, and its corresponding general linear bundle $GL(\cQ_{\bH_{max}})$ on the projective $\bH_{max}$-scheme $\bH_{max}\times E$ is flat over $\bH_{max}$, so the {\it semicontinuity theorem} \cite[Theorem III.12.8]{hrt} applies and the locus
    \begin{equation*}
      \bH_{max\mid ss\mid asym}
      =
      \left\{f\in \bH_{max}\ |\ \dim H^0(GL(\cQ_{\bH_{max}}))\text{ minimal}\right\}
      \subset
      \bH_{max}
    \end{equation*}
    is indeed open.

  \item\label{item:th:asym:indec} {\bf Indecomposable small-rank bundle.} The preceding step, assuming $\cE$ stable, was the base case of an induction on the length $h=\gcd(\rk\cE,\deg\cE)$ in the abelian category of semistable sheaves of slope $\mu(\cE)$. We now  illustrate the induction step at $h=2$, whereupon it will be apparent to the reader how to proceed. The assumption throughout, then, is that $\cE$ is a non-split self-extension of some stable $\cK$:
    \begin{equation}\label{eq:kek}
      0\to \cK\lhook\joinrel\xrightarrow{\quad} \cE\xrightarrowdbl{\quad} \cK\to 0.
    \end{equation}
    For brevity, maximally non-degenerate morphisms with semistable maximally rigid $\cat{KC}$ will be termed {\it rigidifying} (we will mostly be concerned with rigidifying embeddings). 
    
    We know from step \Cref{item:th:asym:stab} that most (i.e. an open dense set) morphisms $\cK\to \cE'$ are rigidifying. The map
    \begin{equation*}
      \Hom(\cE,\cE')
      \xrightarrow{\quad\text{restrict}\quad}
      \Hom(\cK,\cE')
    \end{equation*}
    being onto (by the long exact cohomology sequence), most (again, open dense) morphisms $\cE\to \cE'$ will restrict to rigidifying morphisms $\cK\to \cE'$. Consider one such embedding $\cE\lhook\joinrel\xrightarrow{\iota_{\cE}}\cE'$, restricting to rigidifying $\cK\lhook\joinrel\xrightarrow{\iota_{\cK}}\cE'$. It induces a morphism
    \begin{equation}\label{eq:k2e'modk}
      \cE/\cK\cong \cK
      \lhook\joinrel\xrightarrow{\quad\iota'\quad}
      \cE'/\iota_{\cK}(\cK)
    \end{equation}
    which we know (from its very construction) pulls back the extension
    \begin{equation}\label{eq:ike'e'}
      0\to
      \iota_{\cK}(\cK)
      \lhook\joinrel\xrightarrow{\quad}
      \cE'
      \xrightarrowdbl{\quad}
      \cE'/\iota_{\cK}(\cK)
      \to 0
    \end{equation}
    to the original non-split self-extension of $\cK$. Now, the condition of producing such non-split extensions is open in $\Hom(\cK,\cE'/\iota_{\cK}(\cK))$, hence also, in this case, open {\it dense}. Because furthermore
    \begin{equation*}
      \begin{aligned}
        \Ext^1(\cK,\cK)
        \quad
        \cong
        \quad
        &\End(\cK)^*
          \quad\left(\text{Serre duality \cite[Theorem 3.12]{Huy-FM}}\right)\\
        &\text{is 1-dimensional \cite[Corollary 10.25]{muk-invmod}},
      \end{aligned}         
    \end{equation*}
    {\it all} non-split self-extensions of $\cK$ produce bundles isomorphic to the original $\cE$. Varying \Cref{eq:k2e'modk} over said open dense locus, then, will pull back \Cref{eq:ike'e'} to embeddings $\cE\lhook\joinrel\to \cE'$ with quotients $\left(\cE'/\iota_{\cK}(\cK)\right)/\iota'(\cK)$. By step \Cref{item:th:asym:stab} again we can further assume $\iota'$ rigidifying, and we are done (with the induction step  $h=1 \rightsquigarrow h=2$). 

    The analogue of \Cref{eq:kek} for general $h$ is
    \begin{equation*}
      0\to
      \tensor*[_{h-1}]{\cK}{}
      \lhook\joinrel\xrightarrow{\quad}
      \cE
      \xrightarrowdbl{\quad}
      \cK
      \to 0
    \end{equation*}
    for stable $\cK$ (in the notation of \Cref{re:selfext}), and we can proceed as in upgrading $h=1$ to $h=2$: we still have the crucial $\dim\Ext^1\left(\cK,\tensor*[_{h-1}]{\cK}{}\right)=1$ regardless of $h\ge 2$, by \Cref{eq:selfexts}. 
    
  \item\label{item:th:asym:gen} {\bf The general case.} This is much as the inductive step $h-1 \rightsquigarrow h$ of the preceding step, the induction this time being over the number $s$ of indecomposable summands in
    \begin{equation*}
      \cE\cong \bigoplus_{i=1}^s \cE_i
      ,\quad
      \dim\Hom(\cE_i,\cE_j)
      =
      \dim\Ext^1(\cE_i,\cE_j)
      =0
      ,\quad
      \forall i\ne j
    \end{equation*}
    (the existence of such a decomposition, for a semistable bundle, is precisely what maximal asymmetry means). The statement holds with the indecomposable $\cE_1$ in place of $\cE$ by step \Cref{item:th:asym:indec}. To lift the result to $\cE=\cE_1\oplus \cE_2$ (for the inductive step $s=1\rightsquigarrow s=2$) substitute the {\it split} extension
    \begin{equation*}
      \cE\cong \cE_1\oplus \cE_2
    \end{equation*}
    for \Cref{eq:kek}. The argument is even simpler this time: given an embedding $\cE_1\lhook\joinrel\xrightarrow{\iota_1}\cE'$, {\it any} embedding $\cE_2\lhook\joinrel\xrightarrow{}\cE'/\iota_1(\cE_1)$ will glue to a morphism $\cE\to \cE'$ simply because $\Ext^1(\cE_2,\cE_1)$ is trivial. 
  \end{enumerate}
\end{proof}

One consequence of \Cref{th:asym} is the following improvement on \Cref{cor:stabmidstab}, slackening one stability constraint to basic semistability. 

\begin{theorem}\label{th:prescribedseq}
  Let $\cK$, $\cE$ and $\cF$ be basic semistable bundles, with ranks and determinants satisfying \Cref{eq:slopesrksdets}. 
  \begin{enumerate}[(1),wide]
  \item\label{item:th:prescribedseq:fstab} If $\cF$ is stable then the space of embeddings $\cK\lhook\joinrel\to \cE$ fitting into an exact sequence \Cref{eq:kef} is open dense in $\Hom(\cK,\cE)$.

  \item\label{item:th:prescribedseq:kstab} Dually, if $\cK$ is stable then the space of epimorphisms $\cE\xrightarrowdbl{}\cF$ that can be completed to extensions \Cref{eq:kef} is open dense in $\Hom(\cE,\cF)$. 
  \end{enumerate}
  In particular, if either $\cF$ or $\cK$ is stable then extensions \Cref{eq:kef} exist.
\end{theorem}
\begin{proof}
  \Cref{item:th:prescribedseq:fstab} and \Cref{item:th:prescribedseq:kstab} are mutual duals, the last claim is a consequence thereof, and \Cref{item:th:prescribedseq:fstab} (say) follows from \Cref{th:asym}: if $\gcd(k,n)=1$ there is exactly one charge-$(k,n)$ semistable bundle of given determinant (and it is stable) \cite[Proposition 14 and Lemma 30]{tu}. 
\end{proof}

\begin{remarks}\label{res:notdense}
  \begin{enumerate}[(1),wide]
  \item\label{item:res:notdense:mustbestab} As the two statements \Cref{item:th:prescribedseq:fstab} and \Cref{item:th:prescribedseq:kstab} of \Cref{th:prescribedseq} suggest, in each case it matters crucially which of the bundles $\cK$ and $\cF$ is assumed {\it stable} (as opposed to only {\it semi}stable). Consider, for instance, the following setup: $\cF=\cO(p)$ for some $p\in E$, $\cK$ is the unique non-split self-extension $\tensor*[_2]{\cO}{}$ of $\cO$ (\Cref{re:selfext}), and $\cE$ the unique (because stable) rank-3 bundle of determinant $\cF=\cO(p)$.

    \Cref{th:prescribedseq} then goes through, and furthermore \Cref{th:maxnondeg} shows that the space of embeddings $\cK\lhook\joinrel\xrightarrow{} \cE$ (which will automatically yield quotients isomorphic to $\cF$) has full dimension
    \begin{equation}\label{eq:dimk2e}
      \dim\Hom(\cK,\cE) = \braket{\cK\to\cE} = 2.
    \end{equation}
    The analogue with $\cK$ and $\cF$ interchanged is {\it not} valid: the space of (epi)morphisms $\cE\xrightarrowdbl{} \cF=\cO(p)$ is 2-dimensional, the resulting space of possible kernels is 1-dimensional (because scaling a surjection $\cE\xrightarrowdbl{} \cF$ will not alter its kernel), but $\cE$ contains a {\it unique} subsheaf isomorphic to $\cK=\tensor[_2]{\cO}{}$ (the 2-dimensional automorphism group of $\cK$ canceling out \Cref{eq:dimk2e}).
    
    In short: the space of epimorphisms $\cE\xrightarrowdbl{}\cF$ fitting into a sequence \Cref{eq:kef} {\it cannot} be dense in $\Hom(\cE,\cF)$. Per \Cref{le:chrg31quot} below though, {\it all} epimorphisms $\cE\xrightarrowdbl{}\cF$ will have semistable maximally asymmetric kernels (in accordance with \Cref{th:asym}). 
    
  \item\label{item:res:notdense:quot} To follow up and expand on \Cref{item:res:notdense:mustbestab} above, it is not difficult to describe the Quot scheme \cite[\S 5.1.4]{MR2223407} $\Quot_{\cE/E}^{\cF,\cL} = \Quot_{\cE/E}^{\Phi,\cL}$ consisting of quotients $\cE\xrightarrowdbl{}\overline{\cE}$ whose {\it Hilbert polynomial}
    \begin{equation*}
      \Phi(m):=\chi(\overline{\cE}\otimes \cL^{\otimes m})
      \text{ for some fixed (typically very ample) line bundle }\cL
    \end{equation*}
    (indeed a polynomial by Snapper's \cite[Theorem B.7]{MR2223410}) equals that of $\cF=\cO(p)$.
    
    Suppressing the `$\cL$' superscript (the precise choice of ample $\cL$ making no difference), \Cref{le:chrg31quot} identifies $\Quot_{\cE/E}^{\cO(p)}$ with the symmetric square $E^{[2]}$ of $E$. Recalling once more \cite[\S 14.2.1]{3264} that over smooth projective schemes Chern classes are definable for arbitrary coherent sheaves, consider the ``slice''
    \begin{equation}\label{eq:quotslice}
      \tensor*[_{c_1=\cO(p)}]{\Quot}{_{\cE/E}^{\cO(p)}}
      :=
      \left\{\text{classes of quotients $\cE\xrightarrowdbl{}\overline{\cE}$}\ |\ \overline{\cE}\cong c_1(\overline{\cE})\cong \cO(p)\right\}
    \end{equation}
    of $\Quot_{\cE/E}^{\cO(p)}$ (where we have identified the first Chern class with a line bundle via \cite[Proposition 1.30]{3264}). It classifies quotients of $\cE$ actually isomorphic to $\cF=\cO(p)$ (as opposed to only having the same Hilbert polynomial), and it is identifiable via \Cref{le:chrg31quot} with $\bP^1\cong E/\left(\text{sign change}\right)$: the fiber over $0\in E$ of $E^{[2]}\xrightarrow{\text{addition}}E$.

    
    
  \item\label{item:res:notdense:noo2} The choice of $\cK=\tensor[_2]{\cO}{}$ and $\cF=\cO(p)$ in \Cref{item:res:notdense:mustbestab} above ensures that an indecomposable $\cE$ fitting into an extension \Cref{eq:kef} has $\dim\Gamma(\cE)=\deg\cE=1$ \Cref{eq:tu_lemma-17}. In particular $\cO^2$ does not embed into such $\cE$, showing that the indecomposability requirements of \Cref{th:prescribedseq} are necessary (one as necessary as the other, by dualization). 
  \end{enumerate}  
\end{remarks}

A consequence of \Cref{th:prescribedseq} is the following variant of \cite[Theorem 14.10]{Polishchuk-book} (referred to in \Cref{re:polish.th.14.10} above, which addresses the case $i_0=2$ of \Cref{cor:polish.th.14.10}). It is ``unbiased'' towards the middle term of a short exact sequence, in the sense that it recovers stability for any one of three bundles fitting into an exact sequence from the stability of the other two.

In the statement and proof, in working with bundles $\cE_i$ indexed by indices $i\in \bZ/3=\{0,1,2\}$, we make the convention that $\cE_{3s+i}:=\cE_i[s]$: the $s$-fold {\it translation} \cite[\S II.3.2]{gm_halg_2e_2003} in the bounded derived category $D^b(E)$ of coherent sheaves on $E$. As customary \cite[\S III.5, Definition 3]{gm_halg_2e_2003} in that derived context, $\Hom(\square,\bullet[1]) = \Ext^1(\square,\bullet)$. 

\begin{corollary}\label{cor:polish.th.14.10}
  Let $\cE_i$, $i\in \bZ/3$ be three bundles for which \Cref{eq:slopesrksdets} holds in the present notation:
  \begin{equation*}
    \begin{cases}
      \mu(\cE_0) < \mu(\cE_1) < \mu(\cE_2)\\
      \rk \cE_1 = \rk \cE_0+\rk \cE_2\\
      \det\cE_1 \cong \det\cE_0\otimes \det\cE_2.
    \end{cases}
  \end{equation*}
  If $\cE_{i_0}$ and $\cE_{i_0+1}$ are stable, $\cE_{i_0+2}$ is basic semistable and 
  \begin{equation}\label{eq:ij}
    \dim\Hom(\cE_{i_0},\cE_{i_0+1})=1
  \end{equation}
  then
  \begin{itemize}[wide]
  \item the three fit into a non-split extension
    \begin{equation}\label{eq:e012}
      0\to
      \cE_0
      \lhook\joinrel\xrightarrow{\quad}
      \cE_1
      \xrightarrowdbl{\quad}
      \cE_2
      \to 0,
    \end{equation}
    unique up to scaling in any of the spaces $\Hom(\cE_j,\cE_{j+1})$, $j\in \bZ/3$. 
    
  \item $\cE_{i+2}$ too is stable, and the unique bundle fitting with the other two into a non-split extension \Cref{eq:e012};

  \item and \Cref{eq:ij} holds for all $j\in \bZ/3$ (in addition to the original $i_0$). 
  \end{itemize}
\end{corollary}
\begin{proof}
  This is essentially as sketched in \Cref{re:polish.th.14.10}: by \Cref{th:prescribedseq} extensions \Cref{eq:e012} exist, involving the stable $\cE_{i_0}$ and $\cE_{i_0+1}$ as well as all (mutually distinct \cite[Theorem 10]{Atiyah}) $\cL\otimes \cE_{i_0+2}$ for $\cL$ ranging over the line bundles whose $h^{th}$ power is trivial, for $h:=\gcd(\rk \cE_{i_0+2},\deg\cE_{i_0+2})$. \Cref{eq:ij} then implies that $h=1$, hence stability.

  As for \Cref{eq:ij} holding for all indices, this follows from the fact that the space of morphisms $\cE_0\lhook\joinrel\to \cE_1$ yielding quotient $\cE_2$, up to scaling, can also be identified, dually, with the space (again modulo scaling) of morphisms $\cE_1\xrightarrowdbl{} \cE_2$ with kernel $\cE_0$. 
\end{proof}

\subsection{Some Quot schemes}\label{subse:quot}

We will be interested in several examples of Quot slices $\tensor*[_{c_1=\cF}]{\Quot}{_{\cE/E}^{\cF}}$ (in the notation of \Cref{eq:quotslice}) for (typically semistable) rank-$r$ bundles $\cE$ and line bundles $\cF$. These parametrize either
\begin{itemize}
\item rank-$(r-1)$ subbundles $\cK\subset \cE$ with $\det\cK\cong \cF\otimes (\det\cE)^{-1}$;
  
\item or, equivalently, the corresponding (classes of) quotients
  \begin{equation*}
    \cE\xrightarrowdbl{\quad q_{\cK}}\ol{\cE}\cong \cE/\cK
  \end{equation*}
   with $\ol{\cE}$ of rank 1 and determinant $\cF$. This gives a decomposition
  \begin{equation*}
    \ol{\cE}\cong \cF'\oplus \left(\text{torsion}\cong \cO_D\right)
    ,\quad
    D=\text{some effective divisor}
  \end{equation*}
  and hence an embedding $\cF'\lhook\joinrel\xrightarrow{\iota_{\cK}} \cF$ determined uniquely by the requirement that, when regarded as a section of $\cF\otimes (\cF')^{-1}\cong \cO(D)$, it have zero locus $D$.
\end{itemize}

The rank-2 version of the following preliminary remark appears as \cite[Lemma 3.11]{00-leaves_xv3}; the general statement is no more difficult to prove, so we only record the statement. The symbol `$\dashv$' indicates an adjunction (of functors), with the tail pointing towards the left adjoint (as in \cite[Definition 19.3]{ahs}, say; in this specific instance left vs. right does not make much of a difference). See also \cite[Theorem 3.6.1]{hirz}. 

\begin{lemma}\label{le:deteast}
  For a rank-$r$ vector bundle $\cE$ the epimorphism
  \begin{equation*}    
    \textstyle
    \bigwedge^{r-1}\cE\otimes \cE
    \xrightarrowdbl{\quad}
    \bigwedge^r\cE=\det \cE
  \end{equation*}
  corresponds through the adjunction $(-\otimes \cE)\dashv (-\otimes \cE^*)$ to an $\Aut(\cE)$-equivariant isomorphism $\bigwedge^{r-1}\cE\cong \det \cE\otimes \cE^{*}$.  \qedhere
\end{lemma}

Building on this observation, \Cref{le:wedgesect} is a paraphrase and extension of \cite[Lemma 3.12]{00-leaves_xv3} (the latter, again, handling the rank-2 case). 

\begin{lemma}\label{le:wedgesect}
  Let $\cE$ be a rank-$r$ bundle on $E$ and $\cF$ a line bundle. The morphism sending
  \begin{equation*}
    \left(\cK\lhook\joinrel\xrightarrow{\iota} \cE\right)
    \in
    \tensor*[_{c_1=\cF}]{\Quot}{_{\cE/E}^{\cF}}
  \end{equation*}
  to the composition
  \begin{equation*}
    \begin{tikzpicture}[>=stealth,auto,baseline=(current  bounding  box.center)]
      \path[anchor=base] 
      (0,0) node (l) {$\cE$}
      +(2,.5) node (ul) {$\ol{\cE}$}
      +(4,.5) node (ur) {$\ol{\cE}/\text{torsion}$}
      +(6,0) node (r) {$\cF$}
      ;
      \draw[->>] (l) to[bend left=6] node[pos=.5,auto] {$\scriptstyle q_{\cK}$} (ul);
      \draw[->>] (ul) to[bend left=6] node[pos=.5,auto] {$\scriptstyle $} (ur);
      \draw[->] (ur) to[bend left=6] node[pos=.5,auto] {$\scriptstyle \iota_{\cK}$} (r);
      \draw[->] (l) to[bend right=6] node[pos=.5,auto,swap] {$\scriptstyle $} (r);
    \end{tikzpicture}
  \end{equation*}
  is precisely
  \begin{equation*}
    \begin{aligned}
      \textstyle
      \iota
      \xmapsto{\quad}
      \bigwedge^{r-1} \iota
      &\in
        \Hom\left(\bigwedge^{r-1}\cK,\ \bigwedge^{r-1}\cE\right)
        =
        \Hom\left(\det\cK,\ \bigwedge^{r-1}\cE\right)\\
      &\cong
        \Gamma\left(\bigwedge^{r-1}\cE\otimes (\det\cK)^{-1}\right)
        \cong
        \Gamma\left(\cE^*\otimes \det\cE\otimes (\det\cK)^{-1}\right)
        \quad\text{by \Cref{le:deteast}}\\
      &\cong
        \Gamma\left(\cE^*\otimes \cF\right)
        \cong
        \Hom\left(\cE, \cF\right).
    \end{aligned}    
  \end{equation*}
  \qedhere
\end{lemma}

\begin{lemma}\label{le:chrg31quot}
  Let $\cE$ be a rank-3 degree-1 stable bundle on $\cE$ and $\cF$ a degree-1 line bundle.

  For any ample line bundle $\cL$ on $\cE$ we have
  \begin{equation}\label{eq:le:chrg13quot:quot}
    \Quot_{\cE/E}^{\cF,\cL}
    \cong
    \text{the symmetric square }E^{[2]}.         
  \end{equation}
\end{lemma}
\begin{proof}[sketch]
  The (projective \cite[Theorem 5.14]{MR2223407}) scheme $Q:=\Quot_{\cE/E}^{\cF,\cL}$ parametrizes quotients $\cE\xrightarrowdbl{}\overline{\cE}$ having the same Hilbert polynomial (with respect to $\cL$) as the degree-1 line bundle $\cF$. This determines the rank of $\overline{\cE}$ (namely 1) and its degree (also 1). It follows that $\overline{\cE}\cong \cO(q)$ for some $q\in E$: it must be torsion-free, for otherwise it would have some line-bundle summand of degree $\le 0$ and $\cE$ admits no non-zero maps to any such. The isomorphism \Cref{eq:le:chrg13quot:quot} can now be described as follows.
  
  If $z\ne z'$ the pair $\{z,z'\}\in E^{[2]}$ can be identified with the (unique, up to the aforementioned equivalence) quotient
  \begin{equation*}
    \cE\xrightarrowdbl{\quad}\cO(p-z-z')
    \quad\text{with kernel}\quad
    \cO((z)-(0)) \oplus \cO((z')-(0))
  \end{equation*}
  where $\det\cE = \bigwedge^3\cE \cong \cO(p)$, while the four solutions to $2z=q$ for fixed $q$ correspond to the quotients $\cE\xrightarrowdbl{}\cO(p-q)$ whose kernels are the indecomposable self-extensions 
  \begin{equation*}
    \tensor[_2]{\cL}{_{z}}
    :=
    \cL_{z}\otimes\tensor[_2]{\cO}{}
    ,\quad
    \cL_{z}:= \cO((z)-(0))
  \end{equation*}
  (so in particular $\cL_{0}=\cO$). The $\tensor[_2]{\cL}{_{z}}$, $z\in E$ are all indeed mutually non-isomorphic by \cite[Theorem 10]{Atiyah}.
\end{proof}

\begin{proof}[alternative]
  One may also proceed somewhat differently. Having deduced, as in the preceding argument, that the relevant quotients of $\cE$ are precisely the $\cO(q)$, $q\in E$, we have a fibration
  \begin{equation}\label{eq:qovere}
    Q:=\Quot_{\cE/E}^{\cF,\cL}
    \ni
    (\cE\xrightarrowdbl{\quad}\overline{\cE})
    \xmapsto{\quad}
    \det\overline\cE\cdot (\det\cE)^{-1}
    \in
    \widehat{E}
    \cong E
  \end{equation}
  over the {\it dual} elliptic curve (\cite[\S\S 8, 13]{Mum08}, \cite[\S 9.3]{Polishchuk-book}) of degree-0 line bundles, identified back to $E$ via the ($\cL$-independent \cite[Corollary 9.2]{Polishchuk-book}) map 
  \begin{equation*}
    E
    \ni
    z
    \xmapsto[\cong]{\quad\phi_{\cL}\quad}
    t_z^*\cL\otimes \cL^{-1}
    \cong \cO((z)-(0))
    \in
    \widehat{E}
    ,\quad\cL\in \Pic(E)
    ,\quad deg\cL=1
  \end{equation*}
  ($t_z:=\text{translation by $z$}$) of \cite[\S 6, Corollary 4]{Mum08} or \cite[Corollary 8.6]{Polishchuk-book}.

  The space of epimorphisms $\cE\xrightarrowdbl{}\cO(q)$ is also
  \begin{equation*}
    \Hom(\cE,\cO(q))^{\times}
    :=
    \Hom(\cE,\cO(q))\setminus\{0\},    
  \end{equation*}
  for the stable bundle $\cE$ cannot surject onto line bundles of non-positive degree (and hence slope). Coupled with the isomorphism
  \begin{equation*}
    \textstyle
    \Hom(\cE,\cO(q))
    \cong
    \Gamma(\cO(q)\otimes \cE^*)
    \cong
    \Gamma\left(\cO((q)-(p))\otimes\bigwedge^2\cE\right)
  \end{equation*}
  with last link provided by \Cref{le:deteast} below, \Cref{eq:qovere} is the projectivization $\bP \cV\xrightarrowdbl{}E$ of a vector bundle
  \begin{equation}\label{eq:vbdl}
    \textstyle
    \cV
    ,\quad
    \text{fiber $\cV_z$}
    =
    \Gamma\left(\cO((z)-(0))\otimes\bigwedge^2\cE\right)
    ,\quad
    z\in E
  \end{equation}
  (of rank 2 by \Cref{eq:tu_lemma-17}, $\bigwedge^2\cE$ being semistable \cite[Corollary 2.6.1]{zbMATH03683759} of degree $2\deg \cE=2$). Recall now that
  \begin{itemize}[wide]
  \item the map
    \begin{equation*}
      E^{[r]}
      \xrightarrow[]{\quad\text{addition}\quad}
      E
    \end{equation*}
    is the projectivization of a stable charge-$(r,1)$ bundle on $E$ (claimed in \cite[item 3., p.451]{Atiyah}, with projectivization conventions matching ours, proven in \cite[\S 1(1)]{CaCi93}, where the bundle has degree $-1$ due to dual conventions);

  \item and $\bP\cV'\cong \bP\cV$ as schemes over $E$ {\it precisely} \cite[Exercise II.7.9]{hrt} when $\cV'\cong \cV\otimes \cL$ for some line bundle $\cL$.
  \end{itemize}
  Stable rank-$n$ bundles fall into $n$ equivalence classes under tensoring by line bundles \cite[Theorem 11]{Atiyah}, depending on the degree's residue modulo $n$. It will thus suffice, for our purposes, to show that the rank-2 bundle \Cref{eq:vbdl} is (semi)stable of odd degree (stability follows from semistability \cite[Appendix A, Fact]{tu}, given that the rank and degree are coprime).
  
  To conclude, simply note that the target bundle $\cV$ is the {\it Fourier-Mukai transform} \cite[\S 11.3]{Polishchuk-book} of the stable rank-3 degree-2 bundle $\bigwedge^2\cE$ and hence \cite[Theorem 11.6, (11.3.7) and Lemma 14.6]{Polishchuk-book} has degree $-3=-\rk \bigwedge^2\cE$ (and rank 2 of course, as noted).
\end{proof}

The phenomenon noted in \Cref{res:notdense}\Cref{item:res:notdense:noo2}, whereby $\cE$ has no subbundles isomorphic to the {\it de}composable companion $\cO^2$ of $\cK$ (`companion' in the sense that they have the same {\it associated graded} bundle \cite[preceding Proposition 5.3.7]{lepot-vb}, namely $\cO^2$), is not a universal feature: having fixed stable $\cE$ and $\cF$, it is possible for surjections $\cE\xrightarrowdbl{}\cF$ to have non-isomorphic kernels with the same associated graded bundle, some decomposable and some not.

\begin{example}\label{ex:2o2p}
  Take $\cF$ to be any line bundle of degree 2 and $\cE$ the unique rank-3 stable (so also indecomposable) bundle with $\det\cE\cong \cF$.

  On the one hand, we have an extension \Cref{eq:kef} with $\cK=\tensor[_2]{\cO}{}$ again, by \Cref{th:prescribedseq}. On the other,
  \begin{equation*}
    \dim\Gamma(\cE) = \deg\cE=2\quad\text{\Cref{eq:tu_lemma-17}}
  \end{equation*}
  so there is a canonical {\it embedding}
  \begin{equation}\label{eq:o2ine}
    \cO^2
    \cong
    \Gamma(\cE)\otimes \cO
    \lhook\joinrel\xrightarrow{}
    \cE
  \end{equation}
  whose corresponding quotient $\cE/\cO^2$ cannot have torsion:
  \begin{itemize}[wide]
  \item \Cref{eq:o2ine} is invariant under translation by the 2-torsion $E[2]$ because $\det\cE\cong \cF$ is (being a degree-2 line bundle), so any torsion in $\cE/\cO^2$ would have degree at least 4. But then the torsion-free summand of $\cE/\cO^2$ would have negative degree (and slope), contradicting the (semi)stability of $\cE$. In conclusion, $\cE/\cO^2$ is torsion-free and hence again isomorphic to $\cF$.

  \item Alternatively, because $\dim\Ext^1(\cO_z,\cO)=1$ for any skyscraper sheaf $\cO_z$, some degree-1 line-bundle summand would split off any extension of $\cO_z$ by $\cO^2$ and hence embed into $\cE$; this would again contradict $\cE$'s semistability. 
  \end{itemize}
  Summary: both $\tensor[_2]{\cO}{}$ and $\cO^2$ are realizable as kernels of epimorphisms $\cE\xrightarrowdbl{}\cF$. 
\end{example}

\begin{remark}\label{re:likeext}
  There is of course nothing particularly surprising about \Cref{ex:2o2p}, given the present context of studying the symplectic leaves $L(\cE)$: the phenomenon is very much in line with the already-familiar fact that the isomorphism class of $\cE$ in an extension 
  \begin{equation*}
    0\to \cO\xrightarrow{\quad} \cE\xrightarrow{\quad} \cL\to 0
  \end{equation*}
  (with $\cL$ a fixed line bundle of degree $\ge 3$) can vary within the same associated-graded-equivalence class. \Cref{ex:2o2p} fixes the two rightmost terms $\cE$ and $\cF$ of \Cref{eq:kef} instead (rather than the two outer terms), varying the leftmost (rather than the middle). 
\end{remark}

We remind the reader that a {\it normal} elliptic curve $E\subset \bP^{n-1}$ ($n\ge 3$) \cite[Exercise 7.36]{3264} is one of degree $n$; equivalently, one not contained in any hyperplane (so the notion of normality for curves in $\bP^{n-1}$ extends to arbitrary genus). The case $n=4$ (quartic elliptic curves in $\bP^3$) is discussed in \cite[\S III.2]{Hulek86} (also \cite[Example III.1]{Hulek83}, \cite[Exercises I.5.11 and III.3.6]{hrt} and numerous other sources), and will be of interest below. Every quartic elliptic curve is the intersection of any two of a pencil of quadrics.

Define the {\it secant slices} of a normal (here mostly elliptic) curve $E\subset \bP^{n-1}$ by
\begin{equation}\label{eq:secslice}
  \Sec_{d,z}=\Sec_{d,z}(E)
  :=
  \overline{
    \left\{
      p\in \bP^{n-1}
      \ \big|\
      p\in\text{span of some degree-$d$ divisor with sum $z\in E$}
    \right\}
  }.
\end{equation}
The aforementioned quadrics containing a normal quartic $E\subset \bP^3$ are precisely the secant slices $\Sec_{2,z}$ \cite[\S 1.5(5)]{00-leaves_xv3}: 
\begin{equation}\label{eq:quadriccoinc}
  \cO_{\bP^3}(1)|_{E}\cong \cO\left(\sum_{i=1}^4(z_i)\right)
  \xRightarrow{\quad}
  \left(
    \Sec_{2,z} = \Sec_{2,z'}
    \iff
    z+z'=\sum_{i=1}^4z
  \right),
\end{equation}
so that indeed the family $\{\Sec_{2,z}\}_{z\in E}$ is in effect parametrized by a $\bP^1$ quotient
\begin{equation*}
  E\bigg/\left(z\sim \sum_i z_i-z\right)
  \cong
  \bP^1.
\end{equation*}

Fix, for the remainder of the section, a rank-3 degree-2 stable $\cE$ and a rank-2 line bundles $\cF$. The structure of the quotient scheme $\Quot^{\cF}_{\cE/E}$ ($=\Quot^{\cF,\cL}_{\cE/E}$ for any ample line bundle $\cL$) is intimately connected to the normal-quartic geometric just recalled in brief. It might, for that reason, be instructive to describe that quot scheme (or rather its slices $\tensor*[_{c_1=\cF}]{\Quot}{_{\cE/E}^{\cF}}$ in the notation of \Cref{eq:quotslice}, on which we focus) along the lines of \Cref{le:chrg31quot} and \Cref{res:notdense}\Cref{item:res:notdense:quot}.

By contrast to the degree-1 $\cE$ case (where the relevant quotients were automatically line bundles), rank-1 degree-2 quotients $\cE\xrightarrowdbl{}\ol{\cE}$ come in two flavors:
\begin{equation}\label{eq:2quots}
  \ol{\cE}
  \cong
  \begin{cases}
    \cO((z)+(z')),\quad z,z'\in E&\text{if $\ol{\cE}$ is torsion-free}\\
    \cO(z)\oplus \cO_{z'},\quad z,z'\in E &\text{otherwise}
  \end{cases}
\end{equation}
($\cO_{z'}$ again the skyscraper at $z'$); for once more, as in the proof of \Cref{le:chrg31quot}, the stable positive-degree $\cE$ cannot surject onto bundles of non-positive degree.

The following remark ties this back to the geometry of elliptic normal curves. 

\begin{lemma}\label{le:ellquart}
  Let $\cE$ be a charge-$(3,2)$ stable bundle and $\cF$ a degree-2 line bundle.

  The subspace of $\bP\Hom(\cE,\cF)\cong \bP^{3}$ consisting of (lines through) non-epic morphisms is an elliptic normal (hence quartic) curve $E\subset \bP^3$ with $\cO_{\bP^3}(1)|_E\cong \cF^{\otimes 3}\otimes (\det\cE)^{-1}$.
\end{lemma}
\begin{proof}
  That $\Hom(\cE,\cF)$ is 4-dimensional follows from \Cref{eq:tu_lemma-17}. By the same token
  \begin{equation*}
    \dim\Hom(\cE,\cO(z))
    =
    \dim\Hom(\cO(z),\cF)
    =
    1
    ,\quad
    \forall z\in E
  \end{equation*}
  so there is, for every $z\in E$, a line $\cL_z\in \bP\Hom(\cE,\cF)$ consisting of those morphisms $\cE\to \cF$ factoring through the unique copy of $\cO(z)\subset \cF$. The quartic (meaning degree-4) embedding will be
  \begin{equation}\label{eq:deg4emb}
    E\ni z
    \xmapsto{\quad}
    \cL_z
    \in \bP\Hom(\cE,\cF).
  \end{equation}
  We check that the pullback $\cO_{\bP^3}(1)|_E$ through \Cref{eq:deg4emb} is indeed $\cF^{\otimes 3}\otimes (\det\cE)^{-1}$ (so in particular has degree 4). 

  Because the inverse $\cL:=\cO_{\bP^3}(-1)|_E$ is the pullback of the {\it universal subbundle} \cite[\S 3.2.3]{3264} $\cO_{\bP^3}(-1)$ on $\bP^3$, its fiber at $z\in E$ is precisely $\cL_z$ (hence the notation), and hence it fits into the exact sequence
  \begin{equation*}
    0\to
    \cL
    \xrightarrow{\quad}
    \Gamma(\cF\otimes \cE^*)\otimes \cO    
    \xrightarrowdbl{\quad\text{canonical map}\quad}
    \cF\otimes \cE^*
    \to 0
  \end{equation*}
  on $E$. The middle term is trivial of rank $\dim\Gamma(\cF\otimes \cE^*)=\dim\Hom(\cE,\cF)=4$, while the quotient has rank $\rk\cE=3$ The multiplicativity \cite[Exercise II.5.16(d)]{hrt} of the determinant for short exact sequences of bundles then yields
  \begin{equation*}
    \det\cL^{-1}
    =
    \det(\cF\otimes \cE^*)
    =
    \cF^{\otimes \deg\cE^*}\otimes \det\cE^*
    =
    \cF^{\otimes \rk\cE^*}\otimes (\det\cE)^{-1}
    =
    \cF^{\otimes 3}\otimes (\det\cE)^{-1}
  \end{equation*}
  as desired. 
\end{proof}

To make sense of the following statement, recall the two types of quotients \Cref{eq:2quots} making up the quot-scheme slice $\tensor*[_{c_1=\cF}]{\Quot}{^{\cF}_{\cE/E}}$.

\begin{proposition}\label{pr:chrg32quot}
  Let $\cE$ be a charge-$(3,2)$ stable bundle and $\cF$ a degree-2 line bundle. The map
  \begin{equation}\label{eq:quot2p3}
    \tensor*[_{c_1=\cF}]{\Quot}{^{\cF}_{\cE/E}}
    \ni
    \big(\text{class of }\cE\xrightarrowdbl{\quad}\overline{\cE}\big)
    \xmapsto{\quad}
    \big(\text{line through }\cE\xrightarrowdbl{\quad}\overline{\cE}/\text{torsion}\subseteq \cF\big)
    \in
    \bP\Hom(\cE,\cF)
  \end{equation}
  is isomorphic, as scheme over $\bP^3\cong \bP\Hom(\cE,\cF)$, to the {\it incidence correspondence} attached to the pencil of quadrics containing the elliptic normal curve $E\subset \bP^3$ of \Cref{le:ellquart}. 
\end{proposition}
\begin{proof}
  The claim, formally, is that there is an isomorphism
  \begin{equation*}
    \tensor*[_{c_1=\cF}]{\Quot}{^{\cF}_{\cE/E}}
    =:
    Q
    \cong
    \left\{(t,p)\in \bP^1\times \bP^3\ |\ p\in \text{quadric with parameter }t\right\},
  \end{equation*}
  for the parameter space $\bP^1$ of the family of quadrics, intertwining the map \Cref{eq:quot2p3} and the projection $\bP^1\times \bP^3\to \bP^3$. We will assume $\cF=\det\cE$ to simplify matters (the only additional complications arising in the general case are notational). 
  
  $Q$ parametrizes the trivial-determinant rank-2 subbundles $\cK\subset \cE$, so in particular maps to the moduli space $M(2,0)\cong E^{[2]}$ \cite[Theorem 1]{tu} parametrizing semistable rank-2 degree-0 bundles up to equivalence (in the sense \cite[p.9, preceding \S 3]{tu} of having the same associated graded bundle). Writing $\tensor[^z]{\cL}{}:=\cO((z)-(0))$ (per \Cref{not:miscinit}), that map is
  \begin{equation*}
    Q\ni \cK
    \xmapsto{\quad\psi\quad}
    \{\pm z\}\in E^{[2]}\quad (\text{multiplicity allowed if $2z=0$})
  \end{equation*}
  whenever
  \begin{equation*}
    \cK\cong
    \begin{cases}
      \tensor[^z]{\cL}{}\oplus \tensor[^z]{\cL}{}^{-1}\\
      \text{or the unique self-extension of }\tensor[^{z}]{\cL}{},\quad z\in\text{2-torsion }E[2].
    \end{cases}
  \end{equation*}
  
  The quadrics containing $E\subset \bP^3$ are precisely the level sets of $\psi$.
  
  \begin{enumerate}[(I),wide]
  \item\label{item:pr:chrg32quot:2zne0} {\bf $2z\ne 0$.} Setting $\cK:=\tensor[^z]{\cL}{}\oplus \tensor[^z]{\cL}{}^{-1}$, the map
    \begin{equation*}
      \textstyle
      \Hom(\cK,\cE)
      \ni
      s
      \xmapsto{\quad\text{\Cref{le:wedgesect}}\quad}
      \bigwedge^2 s
      \in
      \Gamma\left(\bigwedge^2 \cE\right)
      \xrightarrow[\cong]{\quad\text{\Cref{le:deteast}}\quad}
      \Gamma(\cF\otimes \cE^*)
      \cong
      \Hom(\cE,\cF)
    \end{equation*}
    induces the {\it Segre embedding} \cite[Exercises I.2.14 and I.2.15]{hrt} of the quadric surface
    \begin{equation*}
      \begin{aligned}
        \psi^{-1}(\{\pm z\})
        &\cong
          \left(\text{embeddings }\cK\lhook\joinrel\xrightarrow{}\cE\right)/\Aut(\cK)\\
        &\cong
          \left(\prod_{w\in \{\pm z\}}\Hom(\tensor[^w]{\cL}{},\cE)^{\times}\right)/\Aut(\cK)\\
        &\cong (\Bbbk^2)^{\times}\times (\Bbbk^2)^{\times}/\Bbbk^{\times}\times \Bbbk^{\times}
          \quad\text{by \Cref{eq:degrk}}\\
        &\cong
          (\bP^1)^2
      \end{aligned}    
    \end{equation*}
    into $\bP^3\cong \bP\Hom(\cE,\cF)$, with the second isomorphism a consequence of the fact that every morphism $\cK\to \cE$ that annihilates neither summand $\tensor[^w]{\cL}{}$, $w=\pm z$ is automatically an embedding: otherwise that morphism would have to factor through a line subbundle of $\cE$ of degree $\le 0$, and no such bundle admits non-zero morphisms from {\it both} $\tensor[^w]{\cL}{}$.

    I claim, specifically, that $\psi^{-1}(\{\pm z\})$ is the secant slice $\Sec_{2,\sigma(D)-z}$ attached to the embedding $E\subset \bP^3$ of \Cref{le:ellquart}, for $\cF\cong \cO(D)$ ($D$ an effective degree-2 divisor) and $\sigma$ as in \Cref{eq:aj}. We analyze embeddings
    \begin{equation*}
      \cK\cong \tensor[^z]{\cL}{}\oplus \tensor[^z]{\cL}{^{-1}}
      \lhook\joinrel\xrightarrow{\quad}
      \cE
    \end{equation*}
    by means of the resulting maps 
    \begin{equation}\label{eq:lw2qbylw}
      \tensor[^z]{\cL}{}
      \lhook\joinrel\xrightarrow{\quad}
      \cE/\tensor[^z]{\cL}{^{-1}}
    \end{equation}
    (note in passing that non-zero morphisms \Cref{eq:lw2qbylw} are in any case embeddings, because $\cE$ has neither torsion nor positive-degree subbundles). The task is to show, having fixed a non-zero morphism (hence embedding) $\tensor[^z]{\cL}{^{-1}}\lhook\joinrel\xrightarrow{}\cE$, that the divisor of
    \begin{equation*}
      \bP^1\cong \bP\Hom\left(\tensor[^z]{\cL}{},\ \cE/\tensor[^z]{\cL}{^{-1}}\right)
    \end{equation*}
    consisting of (lines through) morphisms \Cref{eq:lw2qbylw} whose cokernel has torsion is of degree 2, and that torsion is $\cO_w$ and $\cO_{w'}$ respectively with $w+w'=\sigma(D)-z$. To see this, simply note that for generic choices of the embedding $\tensor[^z]{\cL}{^{-1}}\subset \cE$ we have
    \begin{equation*}
      \cE/\tensor[^z]{\cL}{^{-1}}
      \cong
      \tensor[^{1;p}]{\cL}{}\oplus \tensor[^{1;p'}]{\cL}{}
      \quad\text{in the notation of \Cref{eq:ldd}}
    \end{equation*}
    with $p+p'=\sigma(\det\cE)+z$. 
    
  \item\label{item:pr:chrg32quot:2z0} {\bf $2z=0$.} Similarly, in the four exceptional cases $2z=0$ the quadrics in question are singular with vertices corresponding to the four subsheaves $\tensor[^{\omega}]{\cK}{}:=(\tensor[^{\omega}]{\cL}{})^2\subset \cE$ for $\omega$ ranging over the 2-torsion $E[2]$. There is precisely one such subsheaf for each $\omega$, obtained by tensoring the embedding
    \begin{equation*}
      \cO^2
      \xrightarrow[\cong]{\quad\text{\Cref{eq:tu_lemma-17}}\quad}
      \Gamma(\cE\otimes \tensor[^{\omega}]{\cL}{^{-1}})\otimes \cO
      \lhook\joinrel\xrightarrow{\quad}
      \cE\otimes \tensor[^{\omega}]{\cL}{^{-1}}
    \end{equation*}
    with $\tensor[^{\omega}]{\cL}{}$.
    
    Consider the case $w=0\in E$ (the origin), when $\tensor[^{\omega}]{\cK}{}$ is the non-split self-extension $\tensor*[_2]{\cO}{}$ of $\cO$ (\Cref{res:notdense}\Cref{item:res:notdense:mustbestab}). The map
    \begin{equation*}
      \left(\text{embeddings $\tensor*[_2]{\cO}{}\lhook\joinrel\to \cE$}\right)/\Aut(\tensor*[_2]{\cO}{})
      \xrightarrow{\quad\text{restriction}\quad}
      \bP\Gamma(\cE)
      \cong \bP^1
    \end{equation*}
    is the projection of the pointed cone on $\bP^1$ to that base. We can proceed much as in \Cref{item:pr:chrg32quot:2zne0} above. For a generic morphism $\cO\lhook\joinrel\xrightarrow{\iota} \cE$ the quotient $\cE/\iota(\cO)$ will be $\tensor[^{1;p}]{\cL}{}\oplus \tensor[^{1;p'}]{\cL}{}$ with $p+p'=\sigma(D)$: after dualizing and tensoring with a degree-one line bundle we are back in the setting of \Cref{le:chrg31quot}. The two morphisms $\cO\to \cE/\iota(\cO)$ resulting from embeddings $\tensor*[_2]{\cO}{}\lhook\joinrel\to \cE$ which
    \begin{itemize}[wide]
    \item restrict (up to scaling) to the given $\iota$;

    \item and produce a quotient $\cE/\tensor*[_2]{\cO}{}$ with torsion
    \end{itemize}
    will, respectively, yield
    \begin{equation*}
      \text{torsion}(\cE/\tensor*[_2]{\cO}{})
      \cong
      \cO_p\text{ or }\cO_{p'}
      ,\quad
      p+p'=\sigma(D). 
    \end{equation*}
    Or: the line on the cone on $\bP^1$ which passes through (the point corresponding to) $\iota$ intersects the elliptic curve $E$ of \Cref{le:ellquart} at $p$ and $p'$, with $p+p'=\sigma(D)$. The cone, then, is precisely the singular quadric $\Sec_{2,\sigma(D)}$.
  \end{enumerate}
\end{proof}

\begin{remark}\label{re:ecopiesinquadrics}
  The embedding $E\subset \bP^3\cong \bP\Hom(\cE,\cF)$ in \Cref{le:ellquart} corresponds, in the simplified setting $\cF\cong \det\cE$ of the preceding proof, to the degree-4 line bundle $\cF^{\otimes 2}$. By \Cref{eq:quadriccoinc}, the resulting secant slices coincide in pairs: $\Sec_{2,w} = \Sec_{2,2\sigma(D)-w}$ for $\cF\cong \cO(D)$. This accounts for the shift by $\sigma(D)$ in the identification
  \begin{equation*}
    \psi^{-1}(\{\pm z\})
    \cong
    \Sec_{2,\sigma(D)-z}
  \end{equation*}
  in the proof: the left-hand side is invariant under $z\leftrightarrow -z$, which transformation interchanges the two parameters $\sigma(D)\pm z$ with sum $2\sigma(D)$ (whose corresponding secant slices thus coincide).
\end{remark}

\addcontentsline{toc}{section}{References}

\def\cprime{$'$}

\Addresses

\end{document}